\documentclass[12pt,reqno]{amsart}
\usepackage{graphicx}
\usepackage{subcaption}
\usepackage[utf8]{inputenc}
\usepackage{soul,cancel}
 \usepackage[normalem]{ulem}
\usepackage{color}
\newcommand{\Ell}{\mathscr{L}}
\usepackage[titletoc,title]{appendix}
\usepackage[numbers]{natbib}
\usepackage[T1]{fontenc}
\usepackage[all]{xy}
\usepackage{comment}
\usepackage{amsmath,amssymb,stmaryrd}
\usepackage[british]{babel}
\usepackage{a4wide}
\usepackage[sans]{dsfont}
\usepackage{amsfonts}
\usepackage[mathscr]{euscript}
\usepackage{mathrsfs}
\usepackage{latexsym}
\usepackage{setspace}
\usepackage{amsthm,amsmath}
\usepackage[a4paper,margin=3cm,right=5cm]{geometry}
\usepackage[textsize=scriptsize]{todonotes}
\newtheorem{theorem}{Theorem}[section]
\newtheorem{corollary}{Corollary}[section]

\newtheorem{proposition}{Proposition}[section]
\theoremstyle{definition}
\newtheorem{definition}{Definition}[section]
\theoremstyle{remark}
\newtheorem{remark}{Remark}[section]
\theoremstyle{remark}
\newtheorem{example}{Example}
\theoremstyle{remark}
\newtheorem{assumption}{Assumption}[section]

\numberwithin{equation}{section}

\usepackage{array}
\usepackage{xcolor}
\usepackage{url}
\usepackage{hyperref}
\usepackage{graphicx}
\usepackage{geometry}
\usepackage{enumerate}
\usepackage{empheq}
\usepackage{bbm}
\usepackage{tikz}
\usetikzlibrary{arrows.meta,positioning}

\newcommand{\nc}{\newcommand}

\newcommand{\E}{\mathbb E}
\newcommand{\bN}{\mathbb N}

\renewcommand{\P}{\mathbb P}

\newcommand{\R}{\mathbb R}

\newcommand{\Exp}{\textrm{Exp}}
\newcommand{\Kummer}{_1F_1}
\newcommand{\ind}{\mathds{1}}

\newcommand{\cF}{\mathcal F}

\newcommand{\cS}{\mathcal S}
\newcommand{\minus}{\textrm{-}}
\newcommand{\lambdao}{{\lambda^o}}
\newcommand{\eps}{\varepsilon}
\newcommand{\be}{
\begin{equation}}
  \newcommand{\ee}{
\end{equation}}
\newcommand{\hn}[1]{\begingroup\color{gray}#1\endgroup}

\newenvironment{ber}{\color{violet}}{}
\nc{\bb}{
\begin{ber}}
  \nc{\eb}{
\end{ber}}

\newenvironment{gio}{\color{magenta}}{}
\nc{\bg}{
\begin{gio}}
  \nc{\eg}{
\end{gio}}

\newenvironment{uc}{\color{blue}}{}
\nc{\bu}{
\begin{uc}}
  \nc{\eu}{
\end{uc}}
\title[Markov processes with resetting and cyber security]{A class of Markov processes with resetting and applications to cybersecurity}
\author{Giorgia Callegaro, Umut \c{C}et\.{i}n, Bernardo D'Auria}
\date{\today}

\begin{document}

\maketitle
\begin{abstract}
We introduce a class of piecewise deterministic Markov processes with resetting, motivated by self-exciting models of cyber attacks. Under assumptions reminiscent of ruin theory, we prove the existence and uniqueness of an invariant distribution and derive its density explicitly, thereby establishing ergodicity of the process. We then formulate an associated long-run average control problem in which resetting acts as the intervention mechanism. For exponentially distributed jump sizes, the model becomes analytically tractable, allowing an explicit characterization of the invariant distribution and of the optimal intervention policy.
\end{abstract}

\section{Introduction}\label{sec:introduction}
Cyber attacks have emerged as one of the most consequential operational risks 
facing modern organisations. The scale of the threat is substantial and growing: 
the IBM Cost of a Data Breach Report, based on an analysis of 600 breached 
organisations across 17 industries, found that the global average cost of a 
data breach reached \$4.44 million in 2025, while in the United States 
the figure rose to a record \$10.22 million \cite{IBMBreach2025}. 
At a national level, the UK National Cyber Security Centre reported 
a 50\% year-on-year increase in highly significant incidents for 
the third consecutive year \cite{NCSC2025}, 
and the World Economic Forum's Global Cybersecurity Outlook~2026 
identified AI-related vulnerabilities as the fastest-growing cyber risk, 
with 87\% of surveyed organisations reporting such threats over the course 
of 2025 \cite{WEF2026}. 
These developments have prompted significant policy responses: 
the UK Government published its Cyber Action Plan in January 2026, 
backed by \pounds210 million of central investment \cite{GCAP2026}, 
and introduced the Cyber Security and Resilience Bill to Parliament 
to strengthen regulatory requirements on critical infrastructure providers \cite{CSRBill2025}. 

Against this backdrop,  we develop a novel probabilistic 
framework for the optimal allocation of cybersecurity resources --- 
encompassing both the choice of detection technology and the 
operational policy governing when to intervene --- for an entity 
subject to cyber attacks. This problem is distinct from 
the questions addressed in the existing literature, which has 
concentrated on optimising a continuous investment rate to reduce 
breach probability or on pricing and reserving for cyber-insurance 
contracts, as we discuss in more detail below. 
Our approach is built on a new class of piecewise-deterministic 
Markov processes (PDMP) with endogenous resetting that does not fall into the classical PDMP paradigm of Davis 
\cite{Davis1984, Davis1993} due to the behavior at the boundary and the resetting mechanism, 
and is also fundamentally different from the stochastic resetting 
framework of the mathematical physics literature (see, e.g.,  Evans and Majumdar 
\cite{EvansMajumdar2011, EvansMajumdarSchehr2020}), 
where resetting occurs at exogenous times. 
In our construction, resetting arises endogenously from the dynamics of the attack process reaching a threshold  that is itself the object of optimisation.

Our model incorporates several stylised facts that are 
well-documented in the empirical cybersecurity literature: 
attacks arrive unpredictably with a time-varying frequency --- 
for instance, the Verizon 2025 Data Breach Investigations Report, 
analysing over 22,000 security incidents worldwide, found a 37\% 
year-on-year increase in ransomware-related breaches 
\cite{VerizonDBIR2025}; 
the occurrence of one attack raises the likelihood of subsequent 
attacks in the near future, a phenomenon known as temporal clustering 
or contagion; and the economic losses associated with individual 
attacks are random and potentially very large.

The attack clustering has been documented 
empirically by Baldwin et al.~\cite{Baldwin2017} using data on 
threats to key internet services, and has motivated the use of 
Hawkes processes and related self-exciting point process models 
for describing attack arrivals; see Bessy-Roland et 
al.~\cite{BessyRoland2021} for a multivariate Hawkes framework 
calibrated to US breach data, and the comprehensive survey by 
Awiszus et al.~\cite{Awiszus2023}.

A central question for any entity subject to cyber attacks is how to choose a cybersecurity strategy 
that optimally balances the cost of protective technology against the long-run financial losses from breaches. 
The seminal work of Gordon and Loeb \cite{GordonLoeb2002} introduced an economic framework for this problem 
in a static, single-period setting, showing that optimal security investment should generally not exceed 
a fraction $1/e$ of the expected loss. 
Subsequent works have enriched this framework in various directions: 
Skeoch \cite{Skeoch2021} extends the Gordon--Loeb model to incorporate cyber insurance, 
while Callegaro et al.\ \cite{CFHO2025} develop a continuous-time stochastic extension using a Hawkes process 
for attack arrivals and solve the resulting optimal investment problem via dynamic programming.
Hawkes processes and related point process models have also been used 
in cyber insurance pricing; see Fahrenwaldt, Weber and Weske 
\cite{FahrenwaldtWeberWeske2018} for a network-based approach, 
Bessy-Roland et al.~\cite{BessyRoland2021} for a multivariate Hawkes 
framework applied to cyber insurance, Hillairet and Lopez 
\cite{HillairetLopez2021} and Hillairet et al.~\cite{HLOS2022} 
for counting process and contagion models of cyber risk in insurance portfolios, 
Hillairet, R\'eveillac and 
Rosenbaum \cite{HRR2023} for expansion formulas for Hawkes-driven 
cyber insurance derivatives, and the surveys by 
Awiszus et al.~\cite{Awiszus2023} and He, Jin and Li~\cite{HeJinLi2024}.

In our model, a firm's {\em cybersecurity posture} is characterised by an 
intervention threshold $A > 0$. This parameter reflects both the 
detection capability of the deployed technology and the operational 
policy governing when to respond: a lower threshold corresponds to 
a more sensitive or aggressively configured system that detects and 
responds to attacks at lower intensity levels, while a higher 
threshold represents a less sensitive system or a more permissive 
policy that allows attacks to accumulate before triggering a response. 
The cost $\eta(A)$ of maintaining a given posture is a decreasing convex 
function of $A$, capturing the fact that finer detection and earlier 
intervention require greater investment. 
The firm's objective is to choose the threshold $A$ that minimises the sum of the long-run average cost 
of cumulative losses and the cost of the intervention posture. 
Once the optimal threshold is determined, the resulting stationary risk profile 
provides the natural input for any subsequent insurance pricing or risk transfer analysis.
 
The mathematical framework that we develop to address this question is a new class of 
piecewise-deterministic Markov processes (PDMPs) with \emph{endogenous resetting}. 
The state of the system is described by the triplet $X = (N, \Lambda, C)$, 
where $N$ counts the number of attacks since the last intervention, $\Lambda$ is the stochastic intensity of the attack process, 
and $C$ is the cumulative financial loss. 
Between attacks, the intensity evolves deterministically according to a mean-reverting flow 
with speed $\alpha \geq 0$ and level $\beta > 0$. 
Each attack causes a random jump in the intensity whose size depends on both 
the number of previous attacks and the associated economic damage, 
thereby producing a self-exciting dynamics. 
The key feature of our construction is the resetting mechanism: 
when the intensity $\Lambda$ exceeds the threshold $A$, 
an intervention occurs and the intensity is immediately reset to a baseline level $\lambdao \leq \beta$. 
Crucially, this resetting is \emph{endogenous} --- it is triggered by the process itself reaching a boundary, 
not by an exogenous Poissonian clock as in the stochastic resetting framework 
of Evans and Majumdar \cite{EvansMajumdar2011, EvansMajumdarSchehr2020}. 
As a consequence, the reset times are stopping times for the natural filtration of $X$, 
and the process regenerates at each reset, 
a property that we exploit systematically in the analysis.

The construction of $X$ and the rigorous definition of its dynamics are carried out in Section~\ref{sec:construction} 
via the infinitesimal generator. 
Using the Hille--Yosida theorem, we show in Theorem~\ref{thm:hilleyosida} that the formal  generator is closable and its closure generates a strongly continuous contraction semigroup, thereby establishing the existence  of a PDMP with given infinitesimal characteristics. 
The classical theory of PDMPs was introduced by Davis \cite{Davis1984} 
and further developed in Davis \cite{Davis1993} and subsequent literature (see, e.g.,   Jacobsen \cite{Jacobsen2006} for a recent account of the general theory); 
however, the specific structure considered here --- a self-exciting counting process 
whose intensity is subject to endogenous resetting at a boundary --- 
does not fall within the scope of the existing theory and requires a dedicated analysis (see Remark \ref{r:pdmp} for more details). 
 
 
A natural question is whether the process resets in finite time, 
i.e.,\ whether the intensity ever reaches the level $A$. 
The answer is fundamental to the firm's optimisation problem 
since the ergodic behaviour of $\Lambda$ --- and hence the long-run 
average cost that the firm seeks to minimise --- depends critically 
on whether and how frequently these resets occur.
Section~\ref{sec:resetting} provides necessary and sufficient conditions for this. 
In the case $\alpha = 0$ (no mean reversion), Theorem~\ref{t:alph0finreset} gives a complete characterisation 
via Kolmogorov's three-series theorem applied to the jump sizes of the intensity. 
For $\alpha > 0$, Theorem~\ref{t:harris} establishes that the expected reset time $\mathbb{E}^{\alpha,x}[T_A]$ is finite 
under a condition involving the moment generating functions of the jump sizes 
and a bounded function $r$ satisfying certain re\-cur\-si\-ve inequalities. 
The condition is reminiscent of the classical Cram\'er--Lundberg condition in ruin theory; 
one  key difference is that the jumps in the intensity are state-dependent due to the self-exciting structure.

 One of the principal contributions of this paper is to establish the 
existence of a stationary measure for the process $Y = (N, \Lambda)$ 
and to compute it explicitly. General results in the PDMP literature 
provide abstract criteria for the existence of invariant measures --- 
typically Foster-Lyapunov conditions for an embedded post-jump chain 
(see Dufour and Costa \cite{DufourCosta1999, DufourCosta2008}) --- 
but constructing a suitable Lyapunov function is itself a non-trivial 
model-specific task 
(for a notable case where such a construction is feasible, 
see Chafa\"i, Malrieu and Paroux \cite{ChafaiMalrieuParoux2010} 
in the context of the TCP window-size process), 
and even when existence is established, 
the computation of the stationary density requires a separate analysis. 
In the present paper, the endogenous resetting mechanism provides 
a natural regeneration structure that we exploit, 
via the theory of Harris recurrence \cite{KaspiMandelbaum1994}, 
to establish both existence and uniqueness of the invariant measure 
and to compute its density explicitly. 
This is achieved in full generality under mild assumptions 
on the model parameters, without recourse to Lyapunov functions.

Section~\ref{sec:stationary} is devoted to the stationary measures of the two-dimensional process $Y = (N, \Lambda)$. 
Since the cumulative cost $C$ is increasing, and typically explosive in the limit, the full process $X$ does not possess an invariant measure, 
but the pair $Y$ does under the assumption that the reset time is almost surely finite. 
We appeal to the theory of Harris recurrence for continuous-time Markov processes, following Kaspi and Mandelbaum \cite{KaspiMandelbaum1994}, 
to establish that $Y$ is Harris recurrent with a unique (up to scaling) invariant measure $\Pi$.
 
We compute the invariant measure explicitly by deriving a system of 
ordinary differential equations. The invariant measure $\Pi(0, \cdot)$ for the intensity 
at reset is obtained in explicit form, 
and the measures $\Pi(n, \cdot)$ for $n \geq 1$ are then determined 
recursively via integral expressions involving $\Pi(n-1, \cdot)$. 
The structure of the solution depends on the relationship between 
the mean reversion parameters $\alpha$, $\beta$ and the reset level $\lambda^o$, 
and we treat the relevant cases separately. 
A notable feature of the case $\alpha = 0$ is that the invariant measure 
admits a particularly clean representation in terms of convolutions 
of the jump size distributions, which we also verify via an alternative 
proof based on the embedded Markov chain of $Y$.
 
Section~\ref{sec:exponential} develops a tractable special case in which the jumps in the intensity 
are exponentially distributed, $\lambda^o = \beta$ and $\alpha>0$. 
In this setting, we derive a closed-form expression for the density of the stationary distribution (Theorem~\eqref{thm:stat.dist.alpha>0.beta=lambda0.exp}), which allows the long-run average intensity to be computed in closed form as a function of intervention threshold.  
 
 
The model considered in this paper is related to several strands of the literature. 
The self-exciting structure of the intensity is reminiscent of Hawkes processes, 
which were introduced by Hawkes \cite{Hawkes1971} and have been applied extensively 
in seismology, finance, and more recently in cyber risk modelling. 
However, our process is typically not a Hawkes process: 
the jumps in the intensity depend on both the current state $N$ and the marks $Z_i$, 
and the resetting mechanism has no counterpart in the standard Hawkes framework. 
The resetting is also fundamentally different from the stochastic resetting studied 
in the physics literature following Evans and Majumdar \cite{EvansMajumdar2011}: 
in that framework, resetting occurs at exogenous Poissonian times and the focus is 
on the resulting nonequilibrium stationary state and the optimisation of first passage times. 
In our model, resetting is endogenous and state-dependent, 
occurring precisely when the attack intensity exceeds the intervention threshold 
dictated by the firm's chosen cybersecurity posture.

On the applied side, the closest work to ours is Callegaro et al.\ \cite{CFHO2025}, 
which also uses self-exciting dynamics for cyber attack arrivals and solves an optimisation problem. 
However, the nature of the control is different: 
in \cite{CFHO2025}, following the Gordon--Loeb tradition, 
the control variable is a continuous investment rate that reduces the breach probability, 
and the problem is solved via dynamic programming for a finite-horizon objective. 
In the present paper, the control is a \emph{structural parameter} of the system --- 
the detection threshold --- and the objective is an ergodic (long-run average) criterion. 
The probabilistic analysis required to evaluate this criterion 
(construction of the PDMP, stationary measures, ergodic limits) 
constitutes the main mathematical contribution of the paper. 

The outline of the paper is as follows. 
Section~\ref{sec:SMCR} introduces the model and constructs the PDMP via its infinitesimal generator. 
Section~\ref{sec:resetting} gives conditions for the finiteness of the reset time and its expectation. 
Section~\ref{sec:stationary} computes the stationary measures of the model under the reset-threshold policies. 
Section~\ref{sec:exponential} develops formulates and solves a control problem about the tractable case of exponentially distributed damages.

\section{Self-exciting PDMP with resetting: construction via infinitesimal generator}\label{sec:SMCR}
Our aim in this section is the construction of a piecewise-deterministic Markov process (PDMP) in continuous time that will provide the foundation for the study of cyber-risk mitigation that we shall consider in this paper.

At the heart of our construction lies a counting process ${(N_t^0)}_{t \ge 0}$ with stochastic intensity ${(\Lambda_t^0)}_{t \ge 0}$, which are both defined on a probability space $(\Omega, \mathcal F, \mathbb P)$. Moreover, there will be  a sequence of i.i.d. random variables $(Z_i)_{i\ge 1}$ taking values in $[0,\infty)$ with the common distribution function $G$. We shall assume that the $Z_i$'s have finite second moment and set $\mu_Z:=\int y \,dG(y)$ and $\sigma^2_Z:=\int (y-\mu_Z)^2 \,dG(y)$. We denote the tail distribution by $\bar G(y):=1-G(y)$ and define $z^*:=\inf\{y:G(y)>0\}$.

The attacks arrive at random times $(T_n)_{n\ge 1}$ that are the jump times of a counting process $N^0$. Its associated  intensity process is a solution of the following stochastic differential equation:
\be
\label{eq:lambda.noreset}
d\Lambda_t^0 = \alpha(\beta-\Lambda^0_t)dt + \ell(N_t^0, Z_{N_t^0})\Delta N_t^0, \quad \Lambda_0^0 >0
\ee
where $\ell$ is a strictly positive function such that for each $n$, $z \mapsto \ell(n, z)$ is non-decreasing and continuous and such that $\mathbb P(\ell(n, Z_n) >0) =1$. We set 
\be\label{def.ell_n}
\ell_n := \ell(n, z^*) \ge 0.
\ee

In above, $\beta>0$ is the mean reversion level and $\alpha\geq 0$ is the speed of mean reversion. Clearly, when $\alpha=0$, there is no mean reversion and, consequently, $\Lambda^0$ is constant between jumps of $N^0$.   A simple integration by parts shows the solution $\Lambda^0$  of \eqref{eq:lambda.noreset} is given by
\begin{equation}
  \label{eq:noreset2}
  \Lambda_t^0 = \beta + (\Lambda_0^0 - \beta)  e^{-\alpha t} + \sum_{j=1}^{N_t^0} e^{-\alpha (t - T_j)} \ell(j,Z_j)
\end{equation}
\begin{remark}
a) If $\ell (j,z)=\ell(j)$ and $\alpha=0$, the intensity is an increasing function of~$N^0$. Indeed, in this case $\Lambda_t^0 =h(N_t^0)$, where $h(0)=\Lambda_0^0$ and $h(j)=h(j-1)+\ell(j)$ for $j\geq 1$.

b) If $\ell(j,z)=\ell(z)$, the size of the jumps in the intensity does not depend on~$N^0$ and the resulting counting process is a Hawkes process.
\end{remark}

The marks $(Z_i)_{i\ge 1}$ model the damage to the system associated with each cyber attack. Thus, the cumulative cost of the attacks is given by
\be \label{def:C^0}
C^0_t= \sum_{i=1}^{N_t^0} Z_i.
\ee
A {\em reset} occurs when the intensity of the attacks goes beyond a predetermined level $A>0$. At that point in time, an intervention takes place and  the intensity immediately decreases to $\lambdao\in (0,\beta]$.
Calling $N$ the counting process that is reset to $0$ at each reset time, the reset intensity process, $\Lambda$,  follows the same dynamics of the non-reset process $\Lambda^0$, given in \eqref{eq:lambda.noreset}, between two such instants and it is solution of the following SDE
\begin{equation}
  \label{eq:reset2}
  d\Lambda_t= \alpha(\beta-\Lambda_t)dt + \ell(N_t,Z_{N_t})\Delta N_t.
\end{equation}
The process that traces the total amount of looses, $C$, is not reset at the reset epochs, as it continues to accumulate them. Nevertheless, it is affected by the resets in the sense that at each reset time the intensity of the attacks is reduced implying a reduced rate at which new damages accumulate. 

\begin{remark}
If $\alpha=0$ or $\lambdao=\beta$, $\Lambda$ remains constant after reset until the successive jump.
\end{remark}


Before constructing and studying the Markov process, $X:=(N, \Lambda, C)$,  with resetting, let us provide one motivating example.

\subsection{Motivating example}\label{sec:motivating.example}
Here we consider a firm subject to self-exciting cyber attacks, where  $\alpha=0$ and the intensity $\Lambda_t^{0}=h(N^0_t)$ with $h$ positive and strictly increasing.   This specification, in particular independence of $\Lambda$ from $Z_i$s will allow us to present the main motivations of the paper in a tractable setting with minimal technicalities.


Suppose that this firm is considering investing in a cyber-security technology that intervenes when the intensity of attacks surpasses a fixed threshold $A$. In other words, the security tool cannot (or will not for operational reasons)  detect harmful activity unless the attacks reach a certain frequency.   Since the firm is concerned with long-term risk management, let us suppose that the intensity of attacks is reduced immediately to $\lambdao=h(0)$ at the time of intervention. The long-run average cost of this technology to the firm is given by $\eta(A)$ for some decreasing and convex function $\eta$.

With cyber-security risk mitigation in mind, the objective of the firm is to choose the optimal {\em cybersecurity posture}, i.e., the intervention level $A$, so that the cumulative reported losses on average is minimized together with the cost of the cyber-security technology. More precisely, the firm's objective is to determine the optimal value function, defined as:
\begin{equation}
\label{eq:problem}
V := \inf_{A \in \R_+} J(A)
\end{equation}
where $J(A) := \eta(A) + \lim_{t \rightarrow +\infty}  (C_t/t)$, were $C$ is the total cost under the resetting discipline at level $A>0$.

Denoting by $N$, the process that counts the successful attacks since the last reset time, we have that it is a Markov chain with state space $\{0, 1, \ldots, n^*_A-1\}$, where $n^*_A:= \min\{n\in \mathbb{Z}: h(n)\geq A\}$. 
It is easy to see that $N$ admits a stationary distribution $\pi_A$ and the above control problem can be cast in terms of $\pi_A$. Indeed, Corollary \ref{c:Climit} below shows that
\begin{equation}\label{e:limCex}
\lim_{t \rightarrow +\infty}  \frac{C_t}{t}=\mu_Z \sum_{n=0}^{n^*_A-1}h(n)\pi_A(n).
\end{equation}
The non-null transition rates of $N$ are given by
\[\begin{split}
q(n,n+1) &= h(n), \qquad n\in \{0,1,\ldots, n^*_A-2\} \ ,\\
q(n,0) &= h(n^*_A-1) .
\end{split}\]
By using the local balance equations, it is easy to see that for $n \le n^*_A-1$, the stationary distribution $\pi_A$ satisfies the following relation
\begin{equation}\label{eq:pi_n}
\pi_A(n)=\pi_A(0) \frac{h(0)}{h(n)}.
\end{equation}
and by the normalizing condition we deduce that 
\begin{equation}\label{eq:pi_0}
\pi_A(0)=\left(\sum_{n=0}^{n^*_A-1}\frac{h(0)}{h(n)}\right)^{-1} .
\end{equation}
By substituting the above expression in \eqref{e:limCex}, we have
\[
\lim_{t \rightarrow +\infty}  \frac{C_t}{t}=\mu_Z \pi_A(0)h(0)n^*_A=\frac{\mu_Z n^*_A}{\sum_{n=0}^{n^*_A-1}h^{-1}(n)}
=\frac{\mu_Z}{\E[h^{-1}(I(n^*_A))]},\]
where we denoted by $I(n)$ a random variable uniformly distributed on the discrete set $\{0, 1,\ldots, n-1\}$.  Also note that $\E[I(n)]$ is decreasing in $n$ for $n\geq 1$, because the $I(n)$ are decreasing in stochastic order as a function of $n$ and $h$ is increasing.
Thus, the optimization problem of the firm becomes
\[
\inf_{A \in \R_+} \bigg(\eta(A) + \frac{\mu_Z n^*_A}{\sum_{n=0}^{n^*_A-1} h^{-1}(n)}\bigg),
\]
where the optimiser aims to find a balance between the increasing long-term costs due to security breaches and the decreasing technology expenditures to reduce associated damages - as functions of $A$.

\subsection{The generator of the Markov process \texorpdfstring{$X$}{X}}\label{sec:construction}

To define $X$ uniquely, we proceed by computing its generator $Q^\alpha$. To this end, let $f$ be a bounded and sufficiently smooth test function and  $x=(n,\lambda,z)$. Noting that the probability of the counting process having more than one jump is $o(h)$, we have
\begin{align}
\E[f(X_h)|X_0=x]&= o(h)+ f(n,\lambda +\alpha (\beta-\lambda)h,z) (1-\lambda h) \nonumber \\
&+\lambda h \E[f(n+1,\lambda +\ell(n+1,Z_1), z+Z_1); \lambda +\ell(n+1,Z_1)<A] \nonumber  \\
&+\lambda h \E[f(0,\lambdao, z+Z_1); \lambda +\ell(n+1,Z_1)\ge  A] \nonumber  \\
&=o(h)+ \big(f(x) + \alpha (\beta-\lambda)h f_\lambda(n,\lambda ,z)\big) (1-\lambda h) \nonumber \\
&+\lambda h \E[f(n+1,\lambda +\ell(n+1,Z_1), z+Z_1); \lambda +\ell(n+1,Z_1)< A] \nonumber  \\
&+\lambda h \E[f(0,\lambdao, z+Z_1); \lambda +\ell(n+1,Z_1)\ge A],
\label{eq:E[F(X)]}
\end{align}
where $f_\lambda$ is the partial derivative of $f$ with respect to $\lambda$. Note that the last line in Eq. \eqref{eq:E[F(X)]} describes the state of the process at time of reset when a jump occurs, in which case the process counting the attacks is reset to zero, and the next to last considers the case when the jump is not big enough to pass the threshold $A$. Thus, we can define the generator $Q^\alpha$ by considering
\[
Q^\alpha f(x):=\lim_{h \to 0}\frac{\E[f(X_h)|X_0=x]-f(x)}{h}.
\]

The above calculations lead to
\be \label{e:Qalpha}
\begin{split}
Q^\alpha f(x)&= -\lambda f(x) + \alpha(\beta-\lambda)f_\lambda (x)\\
&+\lambda \int\ind_{\{\ell(n+1,y)<A-\lambda\}} \, f(n+1,\lambda +\ell(n+1,y), z+y) \,dG(y)\\
&+\lambda \int\ind_{\{\ell(n+1,y)\ge A-\lambda\}} \, f(0,\lambdao, z+y) \,dG(y).
\end{split}
\ee
Recall that $G$ is the distribution of the $Z_i$'s. We shall write $Q$ when $\alpha=0$ to ease the notation.

Consider $S:=(\bN \cup\{0\}, [\lambdao,A), [0,\infty))$ and $\bar{S}:=(\bN \cup\{0\}, [\lambdao,A], [0,\infty))$ endowed with their natural topologies. Let $\widehat{C}(\bar{S})$  be the space of continuous functions on $\bar{S}$ that are vanishing at infinity and $\widehat{C}({S}):=\{f|_S: f\in \widehat{C}(\bar{S})\}$. Note that $\widehat{C}(\bar{S})$ is a Banach space when endowed with the sup-norm. We shall also denote the subspaces of functions that are continuously differentiable with respect to the second variable by $\widehat{C}_1(\bar{S})$  and $\widehat{C}_1({S})$, where the derivative at the boundaries  $\{\lambdao\}$ and $\{A\}$ are defined to be the right or left derivatives. The next theorem shows the existence of a strongly continuous contraction semigroup associated with the above generator, which will allow us to rigorously define the process $X$ as a Markov process.

\begin{remark}\label{r:pdmp}
   A fundamental difference between the PDMP constructed in this paper 
and those in the existing theory pioneered by Davis \cite{Davis1984} 
lies in the boundary behaviour. In the classical framework of 
\cite{Davis1984} and the literature that builds upon it, 
transitions are permitted within the interior of the state space 
as well as between interior points and boundary points, 
but not from one boundary point to another. 
Our construction, however, requires that $\Lambda$ be instantaneously 
sent to the boundary point $\lambda^o$ upon reaching the boundary 
point $A$. To the best of our knowledge, this type of 
boundary-to-boundary transitions has not been treated in the existing 
PDMP literature.
\end{remark}

\begin{theorem}\label{thm:hilleyosida}
$Q^\alpha$ with domain $\mathcal{D}(Q^\alpha):=\widehat{C}_1({S})$  is a closable linear operator on $\widehat{C}({S})$ and its closure is the infinitesimal generator of a strongly continuous contraction semigroup $(P_t)_{t\geq 0}$ on $\widehat{C}({S})$.
\end{theorem}
\begin{proof} The detailed proof is provided in Appendix \ref{app:thm:hilleyosida}. 
\end{proof}

\subsection{PDMPs with resetting}
\begin{definition}\label{d:semcr}
Given constants $\alpha\geq 0, \beta> 0$,  $A\in (\beta,\infty)$, and $\lambdao \in (0,\beta]$ and a distribution function $G$,  a self-exciting piecewise-deterministic Markov process (PDMP) with resetting is a Markov process in continuous time with state space $S:=(\{\bN \cup\{0\}\}\times [\lambdao,A)\times [0,\infty))$ that is endowed with  a $\sigma$-algebra $\cS$, and   generator $Q^\alpha$ given  by \eqref{e:Qalpha}.
\end{definition}
In the sequel, $X$ will denote a self-exciting PDMP with resetting for given $(\alpha,\beta, \lambdao, A, G)$. Its transition semigroup and infinitesimal generator will be given by $(P_t)_{t\geq 0}$ and $Q^\alpha$.

For each $x\in S$, $\P^{\,\alpha,x}$ will denote its law induced on the space of $S$-valued right continuous paths when $X_0=x$. The expectation operator associated to this law will be  $\E^{\alpha,x}$. We shall drop $\alpha$ from the superscript when $\alpha=0$ to ease the notation. If $\mu$ is a measure on $(S,\cS)$, $\P^{\,\alpha,\mu}$ will denote the law of $X$ with initial law  $\mu$. $(\cF^0_t)_{t \geq 0}$ will denote the natural filtration of $X$ and $(\cF_t)_{t \geq 0}$ its minimal right continuous augmentation.

The following result establishes the Doob-Meyer decomposition for the individual components of $X$. To this end, it is useful to introduce another distribution function
\be \label{e:Gell}
G_{\ell}(n,a) := G(\ell^{-1}(n,a)),
\ee
where
\begin{equation}\label{def:ell.inv}
\ell^{-1}(n,a):=\inf\{z\geq  0: \ell(n,z)>a\},
\end{equation}
with the convention that $\inf \emptyset=\infty.$
Note that $G_{\ell}(n,\cdot)$ is the distribution function of $\ell(n, Z_n)$ for $n\geq 1$.   Indeed, since $\ell^{-1}$ is non-decreasing and $\ell(n,\ell^{-1}(n,a))=a$  by the continuity of $\ell$, one observes that    $Z_n \leq \ell^{-1}(n,a)$ implies $\ell(n,Z_n)\leq a$. Conversely, if $Z_n >\ell^{-1}(n,a)$, then $\ell(n,Z_n)\geq a$. Moreover, if $\ell(n,Z_n)=a$, then $Z_n\leq \ell^{-1}(n,a)$ by the definition of inverse since $\ell$ cannot decrease. But this contradicts $Z_n >\ell^{-1}(n,a)$. Thus, $\ell(n,Z_n)> a$  when $Z_n >\ell^{-1}(n,a)$, too.
\begin{theorem}\label{thm:mart}
Let $X=(N,\Lambda,C)$  be a self-exciting PDMP with resetting. Then, the following processes are $({(\cF_t)}_{t \ge 0}, \P^{\,\alpha,x})$-martingales for each $x\in S$:
\begin{align}
M^{N}_t&= N_t -\int_0^t\Lambda_s ds + \int_0^t (N_s+1)\Lambda_s {\bar{G}_{\ell}(N_s+1, A-\Lambda_s)}ds; \label{e:MN}\\
M^{\Lambda }_t&= \Lambda_t- \int_0^t \big(\alpha(\beta-\Lambda_s)+  L^o(N_s+1, \Lambda_s)\big)ds,\label{e:ML}\\
M^{C}_t&= C_t-\mu_Z \int_0^t\Lambda_s ds,\label{e:MC}
\end{align}
where $\bar{G}_\ell=1-G_\ell$,  and
\begin{equation} \label{e:Lo}
L^o(n,\lambda):= \lambda (\lambdao - \lambda) \bar{G}_{\ell}(n, (A-\lambda)\minus) +  \lambda \int\ind_{\{y<A-\lambda\}} \, y \,dG_{\ell}(n, y).
\end{equation}
\end{theorem}
\begin{proof}
Let $M^N_t:=N_t - \int_0^t Q^\alpha f_1(X_s)ds$, $M^\Lambda_T:=\Lambda_T - \int_0^t Q^\alpha f_2(X_s)ds$, and $M^C_t:=C_t - \int_0^t Q^\alpha f_3(X_s)ds$, where $f_1(x)=n, f_2(x)=\lambda,$ and $f_3(x)=z$ for $x=(n,\lambda,z)$. By the definition of generator,  these processes are martingales.

Next observe that,
since $G_{\ell}(n,u) = G(\ell^{-1}(n,u))$,
\[
Q^\alpha f_1(x)= -\lambda n + \lambda (n+1) G_{\ell}(n+1, (A-\lambda) \minus)=\lambda - \lambda (n+1) \bar{G}_{\ell}(n+1, (A-\lambda) \minus).
\]
Similarly,
\begin{align*}
Q^\alpha f_2(x)&= -\lambda^2+ \alpha(\beta-\lambda) +\lambda  \int\ind_{\{\ell(n+1,y)<A-\lambda\}} \, (\lambda + \ell(n+1,y))\,dG(y)\\
&+\lambda \lambdao \bar{G}_{\ell}(n+1, (A-\lambda)\minus)\\
&=\alpha(\beta-\lambda) +\lambda  \int\ind_{\{\ell(n+1,y)<A-\lambda\}} \, \ell(n+1,y) \,dG(y)\\
&+\lambda (\lambdao - \lambda) \bar{G}_{\ell}(n+1, (A-\lambda)\minus)\\
&=\alpha(\beta-\lambda) +\lambda (\lambdao - \lambda) \bar{G}_{\ell}(n+1, (A-\lambda)\minus) +\lambda  \int\ind_{\{y<A-\lambda\}} \, y \,dG_{\ell}(n+1, y) \\
&=\alpha(\beta-\lambda) + L^o(n+1,\lambda)
\end{align*}
as well as
\[
Q^\alpha f_3(x)= -\lambda z + \lambda z + \lambda \int y \,dG(y).
\]
\end{proof}
\begin{proposition}\label{C.ergodic} Let $X=(N,\Lambda,C)$  be a self-exciting PDMP with resetting and $M^C$ be the martingale in \eqref{e:MC}. Then, for each $x\in S$,
\[
\E^{\alpha,x}[(M_t^C)^2]=\int_0^t \E^{\alpha,x}[\Lambda_s] (\mu_Z^2+\sigma^2_Z) ds\leq A(\mu_Z^2+\sigma^2_Z)t.
\]
In particular, $\frac{M^C_t}{t}\to 0$
in $L^2(\P^{\,\alpha,x})$ and $\P^{\,\alpha,x}$-a.s.
\end{proposition}
\begin{proof}
Using $f(x)=z^2$ whenever $x=(n,\lambda,z)$, one can show that
\[
Q^\alpha f(x)= \lambda \mu_Z(2 z+ \mu_Z) + \lambda \sigma^2_Z.
\]
Thus, $R_t:= C_t^2-\int_0^t \Lambda_s \{\mu_Z(2C_s +\mu_Z) +\sigma^2_Z\}ds$ is a martingale.  Since $\Lambda$ is bounded due to resetting, this in turn implies $M^C$ is a square integrable martingale.
Next observe that
\begin{align*}
(M_t^C)^2 &= C_t^2 + \mu_Z^2 \Big(\int_0^t\Lambda_sds\Big)^2-2C_t \mu_Z\int_0^t\Lambda_sds\\
&=C_t^2 + \mu_Z^2 \Big(\int_0^t\Lambda_sds\Big)^2-2\mu_Z\int_0^tC_s\Lambda_sds-2\mu_Z\int_0^t \Big(\int_0^s \Lambda_rdr \Big)dM^C_s\\
&-2\mu_Z^2 \int_0^t \Big(\int_0^s \Lambda_rdr\Big)\Lambda_s ds\\
&=R_t- 2\mu_Z\int_0^t \Big(\int_0^s \Lambda_rdr \Big)dM^C_s
+\int_0^t \Lambda_s (\mu^2_Z +\sigma^2_Z)ds.
\end{align*}
In particular, the angle bracket process for $M^C$ is given by
\[
\langle M^C\rangle_t=\int_0^t \Lambda_s (\mu^2_Z +\sigma^2_Z)ds.
\]
Thus, due to the boundedness of $\Lambda$ and the square integrability of $M^C$, we have
\[
\E^{\alpha,x}[(M_t^C)^2]= \int_0^t \E^{\alpha,x}[\Lambda_s] (\mu_Z^2+\sigma^2_Z)ds \le
A (\mu_Z^2+\sigma^2_Z) t
\]
and this proves the first claim. The convergence in $L^2$ to $0$ is immediate, which also implies convergence in probability.

To show almost sure convergence, note that $\langle M^C\rangle_\infty=\infty$ since $\Lambda \ge\lambdao>0$. Moreover, for any sequence $t_n$ increasing to $\infty$, $(M_{t_n}^C)_{n\geq 1}$ is a martingale with angle bracket process  $(\langle M^C\rangle_{t_n})_{n\geq 1}$. Thus,
\[
\frac{M_{t_n}^C}{\langle M^C\rangle_{t_n}}\to 0,
\]
as $n\to \infty$ (see Durrett \cite[Theorem 4.5.3]{durrett}). Therefore,
\[
\frac{M_{t}^C}{\langle M^C\rangle_{t}}\to 0, \quad t\to \infty.
\]
Since $\langle M^C\rangle_{t}\leq K t$ for some constant $K$ and for each $t$, the claim follows.
\end{proof}
Now, the following result is immediate.
\begin{corollary}\label{c:Climit}
Let $X=(N,\Lambda,C)$  be a self-exciting PDMP with resetting. Then, for each $x\in S$,
\[
\lim_{t\to \infty} \frac{C_t}{t}=\mu_Z\lim_{t\to \infty}\frac{\int_0^t\Lambda_sds}{t}, \; \P^{\,\alpha,x}\mbox{-a.s.}
\]
\end{corollary}
Corollary  \ref{c:Climit} reveals that the long-run average cost of cyber attacks is governed by the time-average of the intensity process $\Lambda$. In the presence of resetting, this time-average is in turn determined by the stationary distribution of $\Lambda$, making its analysis central to the optimisation problem \eqref{eq:problem}. The remainder of the paper is largely devoted to the ergodic behaviour of the process  $(N,\Lambda)$. As a necessary first step, we must establish that the reset time is finite, which is the subject of the next section.
\section{Sufficient conditions for resetting in finite time}\label{sec:resetting}

Whether resetting event occurs in finite time is crucial to understand the limiting behaviour of the intensity process. We shall show in Section \ref{sec:stationary} that the pair $(N,\Lambda)$ is ergodic under the assumption that the reset occurs in finite time. This section is devoted to finding necessary and sufficient conditions for the finiteness of the reset time, $T_A:=\inf\{t>0: \Lambda_t=\lambdao, \Lambda_{t-}\in (\lambdao,A)\}$. Note that the first reset time coincides with the first hitting time of $[A,\infty)$ by the process $\Lambda^0$, this will be useful in the proof of Theorem 4.2.

One can expect that a reset will not happen if the level $A$ is large and the jumps in the intensity process are relatively small.

\begin{example}\label{e:noreset}
Consider the $dG_{\ell}(n, u)=2^n \ind_{\{0\leq u \leq 2^{-n}\}}$;  that is, the $n$th-jump is uniformly distributed on the interval $[0,2^{-n}]$.
Thus, 
$$\sum_{n = 1}^{N_t^0} \ell(n, Z_{n}) \leq \sum_{n\ge1} 2^{-n} \leq 1.
$$
Assuming that $\lambdao=\beta$ and $A > \Lambda_0+1$,  we have that
$$
\Lambda_t \le \beta + \sum_{n=1}^{N_t^0} \ell(n, Z_{n}) < A
$$
since the process $\Lambda$ is always non-increasing between two jumps whenever $\Lambda_t \ge \beta$.
Hence, the intensity never resets back to $\lambdao$.
\end{example}

If $\alpha=0$, a necessary and sufficient condition for the finiteness of the reset time $T_A$ readily follows from Kolmogorov's three series theorem.
\begin{theorem}\label{t:alph0finreset}
Suppose $\alpha=0$. Then, $\P^x(T_A<\infty)=1$ for all $A$ if and only if one of the following conditions is violated for some $A<\infty$:
\begin{enumerate} \setlength{\itemsep}{0.75em}
\item $\sum_{n=1}^\infty \P^x(\ell(n,Z_n)>A)<\infty$.
\item $\sum_{n=1}^\infty \E^x[\ell(n,Z_n);\ell(n,Z_n)>A)]<\infty$.
\item $\sum_{n=1}^\infty \mathbb{V}^x[\ell(n,Z_n);\ell(n,Z_n)>A)]<\infty$,
where $\mathbb{V}^x$ stands for variance under $\P^x$.
\end{enumerate}
\end{theorem}
Observe that the distribution of $Z_i$'s is independent of $x$; thus, it suffices to check the above conditions for  some $x$.
\begin{proof}
Note that $\P^x(T_A<\infty)=1$ for all $A$ if and only if $\sum_{n=1}^\infty\ell(n,Z_n)=\infty,\, \P^x$-a.s. Then the claim follows from Kolmogorov's three series theorem.
\end{proof}

It is natural to ask whether the reset time has finite mean. The following result establishes that, under a familiar condition on the jump distributions, the expected reset time is indeed finite for any $A$.
\begin{theorem} \label{t:harris} Suppose $\alpha >0$ and let  define
\begin{equation}\label{e:defphi}
\varphi_n(t)=\int\exp(t\,\ell(n,z))\,dG(z).
\end{equation}
Suppose that there exists a bounded function $r:\bN\cup\{0\} \to \R$ that does not change sign and  is bounded away from $0$, and satisfies
\be \label{eq:adjusmenteq}
\begin{split}
\varphi_{n+1}(r(n+1))\geq 1+ \alpha r(n), \quad \forall n\geq 0, \quad \mbox{if } r(0)>0;\\
\varphi_{n+1}(r(n+1))\leq 1+ \alpha r(n), \quad \forall n\geq 0, \quad \mbox{if } r(0)<0.
\end{split}
\ee
Then,
if $\sup_n \varphi_n(r(n))<\infty$ and $\sup_n r(n)<\infty$, $\E^{\alpha,x}[T_A]<\infty$.
\end{theorem}
\begin{proof}
We shall prove the claim when $r$ is a positive function as the other case follows from analogous arguments.

Consider the system without resetting, i.e., $Y^0:=( N^0, \Lambda^0)$, where $\Lambda^0$ is given by \eqref{eq:noreset2}.  Define an operator
\be \label{e:Ralpha}
\begin{split}
R^\alpha f(x)&= -\lambda f(x) + \alpha(\beta-\lambda)f_\lambda (x)\\
&+\lambda \int f(n+1,\lambda +\ell(n+1,y)) \,dG(y),
\end{split}
\ee
and consider the function
\[
f(n,\lambda):=\exp(r(n)(\lambda-\lambdao)).
\]
With $T_A=\inf\{t\geq 0: \Lambda_t^0\geq A\}$, an application of It\^o's formula yields
\[
\begin{split}
f(N^0_{T_A\wedge t}, \Lambda_{T_A\wedge t}^0)&= f(N^0_0, \Lambda_0^0) + \int_0^{T_A\wedge t} r(N^0_s)\alpha(\beta-\Lambda_s^0) f(N^0_s,\Lambda_s^0)ds \\
&+ \sum_{0<s\leq t\wedge T_A}\big(f(N^0_{s}, \Lambda_s^0)-f(N^0_{s-},\Lambda_{s-}^0)\big).
\end{split}
\]

Thus,
\[
f(N^0_{T_A\wedge t}, \Lambda_{T_A\wedge t}^0)= f(N^0_0, \Lambda_0^0) + M_{t\wedge T_A}+\int_0^{T_A\wedge t}R^{\alpha}f(N^0_s,\Lambda_s^0)ds,
\]
where $M$ is a local martingale.
Next,  observe that
\begin{align*}
&\kern-1em\E[\exp(r(N^0_{T_A})\Lambda_{T_A}^0)\ind_{\{T_A<\infty\}}]\\
&\leq \E\big[\exp\left(r(N_{T_A-}^0+1)(A+ \ell(N_{T_A-}^0+1, Z_{N^0_{T_A-}+1})\right)\ind_{\{T_A<\infty\}}\big]\\
&\leq K\E\big[\exp\left(r(N_{T_A-}^0+1)\ell(N_{T_A-}^0+1, Z_{N^0_{T_A-}+1})\right)\ind_{\{T_A<\infty\}}\big]\\
&=K \E[{\varphi}_{N^0_{T_A-}+1}\left(r(N^0_{T_A-}+1)\right)],
\end{align*}
where $K$ is some finite constant.

Since $\sup_n \varphi_n(r(n))<\infty$ and $r$ is bounded, it follows that the family of random variables $f(N^0_{T_A\wedge \tau}, \Lambda_{T_A\wedge \tau}^0)_{\tau}$, where $\tau$ is ranging over all finite stopping times, is uniformly integrable. This in turn implies $M^{T_A}$ is a submartingale bounded from above by an  integrable random variable, and thus $\E[M_{T_A}]\geq 0$ by Fatou's lemma. Consequently,
\[
\begin{split}
  \E[f(N^0_{T_A}, \Lambda_{T_A}^0)]&\geq  f(N^0_0, \Lambda_0^0) +\E\Big[\int_0^{T_A}R^{\alpha}f(N^0_s,\Lambda_s^0)ds\Big]\\
  &\geq f(N^0_0, \Lambda_0^0) +\E\Big[\int_0^{T_A} \alpha \beta r(N^0_s)f(N^0_s,\Lambda_s^0)ds\Big]\\
  &\geq f(N^0_0, \Lambda_0^0) + \eps \E[T_A]
\end{split}
\]
since $r$ is bounded away from $0$. Finally,
\begin{align*}
\E[f(N^0_{T_A}, \Lambda_{T_A}^0)]&=\E[f(N^0_{T_A}, \Lambda_{T_A}^0)\ind_{\{T_A<\infty\}}]+\E[f(N^0_{T_A}, \Lambda_{T_A}^0)\ind_{\{T_A=\infty\}}]\\
& \leq K\E[{\varphi}_{N^0_{T_A-}+1}\left(r(N^0_{T_A-}+1)\right)]+ \exp(\|r\|_\infty(A-\lambdao))<\infty.
\end{align*}
This completes the proof.
\end{proof}
\begin{remark}\label{r:harris}
Suppose that $\ell$ does not depend on $n$ and $\alpha>0$. Then, one can expect $r$ satisfying \eqref{eq:adjusmenteq} to be single valued. Indeed, if $\E[\ell(Z_1)]>\alpha$, there exists a unique $r^*<0$ solving \eqref{eq:adjusmenteq} with equality for each $n$.

On the other hand, if $\E[\ell(Z_1)]\leq \alpha$ and the function $\varphi$ is finite in some open  neighbourhood of $0$ in which $\varphi'$ exceeds $\alpha$, there is always a strictly positive $r$ so that \eqref{eq:adjusmenteq}  is satisfied and $\varphi(r)<\infty$.
\end{remark}
\begin{remark} Note that $\varphi_n(r(n))\leq 1$ for all $n$ when $r$ is a negative function. Thus, the condition on $\varphi$ is imposed only if the functions satisfying the conditions of the theorem are positive.

On the other hand, one can always find negative functions satisfying \eqref{eq:adjusmenteq} if $(\ell(n,Z_n))_{n\geq 1}$ are non-degenerate; that is, $\varphi_n$ is not identically $1$. Indeed, by setting $r(0)=-1/(2\alpha)$, one can recursively compute $r(n)$ for $n\geq 1$ satisfying \eqref{eq:adjusmenteq} with equality. However, there is no guarantee that the  thus obtained function~$r$ is bounded away from $0$ (see Example \ref{e:noreset}).
\end{remark}
\begin{remark}
Suppose that $\ell$  is multiplicative; that is, $\ell(n,z)=\ell(n)z$ for all~$n$. Assume further that $\ell(n)>\delta>0$ for each~$n$. Let us denote the moment generating function of $Z_i$ by $M_Z$. Then, $\varphi_n(r)= M_Z(r\ell(n))$. Suppose that there exists an $r^*\neq 0$ such that $M_Z(\delta r^*)=1+ \alpha r^*$. If $r^*>0$, we have $\varphi_n(r^*)\geq 1+ \alpha r^*$ since $M_Z$ is increasing. By the same token, if $r^*<0$, $\varphi_n(r^*)\leq 1+ \alpha r^*$. Thus, in both cases setting $r(n)=r^*$ for all $n$, we obtain a function satisfying the conditions of Theorem \ref{t:harris}.
\end{remark}

The next result treats the case $\alpha=0$. Its proof runs along similar lines and is, therefore, left to the reader.
\begin{theorem}\label{t:harris0} Suppose $\alpha =0$, let $T_A=\inf\{t\geq 0: \Lambda_t^0\ge A\}$. Consider a function $r:\bN\cup\{0\} \to \R$ that does not change sign and satisfies
\be \label{eq:adjusmenteq0}
\begin{split}
  \varphi_{n+1}(r(n+1))> 1, \quad \forall n\geq 0, \quad \mbox{if } r(0)>0;\\
  \varphi_{n+1}(r(n+1))< 1, \quad \forall n\geq 0, \quad \mbox{if } r(0)<0.
\end{split}
\ee
where $\varphi$ is as in \eqref{e:defphi}. Then, $\E^x[T_A]<\infty$ if both the following conditions hold:
\begin{enumerate} \setlength{\itemsep}{0.75em}
  \item $\sup_n \varphi_n(r(n))<\infty$ and $\sup_n r(n)<\infty$, and
  \item $\inf_{n\geq 1}|\varphi_n(r(n))-1|>0$.
  \end{enumerate}
\end{theorem}


\section{Stationary measures for a self-exciting PDMP with resetting} \label{sec:stationary}
Let $X$ be a self-exciting PDMP with resetting. Since $C$ is increasing, $X$ will not possess an invariant distribution. Nevertheless, if one considers  $Y=(N, \Lambda)$, it is possible to show the existence of a stationary measure that is unique up to a scaling factor under suitable conditions.


We shall denote the invariant measure of $Y$
by $\Pi$. To find an equation that $\Pi$ satisfies, it will be helpful to consider the following representation for its generator, still denoted by $Q^\alpha$ to ease notation.
\begin{proposition}  Let $X=(N,\Lambda,C)$  be a self-exciting PDMP with resetting and $Y=(N,\Lambda)$. The generator of  $Y$ is given by
  \be \label{e:QalphaY}
  \begin{split}
    Q^\alpha f(n, \lambda)
    =& -\lambda f(n, \lambda) + \alpha(\beta-\lambda)f_\lambda (n, \lambda)\\
    &+\lambda \int\ind_{\{u < A-\lambda\}} f(n+1,u+\lambda) \,dG_{\ell}(n+1, u)\\
    &+\lambda f(0,\lambdao) \bar{G}_{\ell}(n+1, (A-\lambda) \minus),
  \end{split}
  \ee
  where  $G_{\ell}(n, \cdot)$ is the distribution function defined in \eqref{e:Gell}.
\end{proposition}
We shall denote the restriction of $S$ to $\{\bN \cup\{0\}\}\times [\lambdao,A)$ by $S_Y$ and the corresponding $\sigma$-algebra by $\mathcal{S}_Y$. To ease the notation we shall still denote the associated semigroup by $(P_t)_{t\geq 0}$.


\
\begin{theorem}
  Suppose that $\P^{\,\alpha,y}(T_A<\infty)=1$ for all $y\in S_Y$. Then,  $Y$ is Harris recurrent with unique (up to scaling) invariant measure given by
  \[
    \Pi(B)= \E^{\alpha,y^*}\int_0^{T_A} \ind_B(Y_t)dt,
  \]
  where $y^*=\{0\}\times\{\lambdao\}$. Moreover, $\Pi$ is finite if and only if $\E^{\alpha,y^*}[T_A]<\infty$.
\end{theorem}
\begin{proof}
  The claim  follows from the example at the bottom of p. 212 in \cite{KaspiMandelbaum1994}. Indeed, $T_A$ is a terminal time and $\P^{\,\alpha,y^*}(T_A>0)=1$ by the definition of $T_A$. Moreover, if we denote the natural filtration of $Y$ by $(\cF^Y_t)_{t\geq 0}$ and set $\cF^Y_\infty:=\vee_{t\geq 0}\cF^Y_t$, we immediately have
  \[
    \E^{\alpha,y}[\mathcal{Z}\circ \theta_{T_A}|\cF^Y_{T_A-}]=\E^{\alpha,y^*}[\mathcal{Z}], \qquad \forall y\in S_Y, \,\mathcal{Z}\in \cF^Y_\infty,
  \]
  where $\theta$ is the shift operator since the process regenerates at $T_A$.
\end{proof}

We will next compute the invariant measure for  $Y$ by considering three cases separately: 
$(\alpha>0\ \&\  \beta>\lambdao)$, 
$(\alpha>0\ \&\ \beta=\lambdao)$, 
and $\alpha=0$.
Throughout this section the following assumption will be enforced:
\begin{assumption} \label{a:finitereset} $\P^{\,\alpha,y}(T_A<\infty) = 1$ for all $y\in S_Y$.
\end{assumption}
Recall that Theorems \ref{t:alph0finreset}, \ref{t:harris} and \ref{t:harris0} give necessary and sufficient conditions for the finiteness of $T_A$ and its expectation.


\begin{theorem}\label{thm:stat.dist.alpha>0.beta>lambda0}  Let $X=(N,\Lambda,C)$  be a self-exciting PDMP with resetting and $Y=(N,\Lambda)$. Under Assumption \ref{a:finitereset} holds and with $\alpha>0\ \&\  \beta>\lambdao$, $Y$ has invariant measure $\Pi$ on $\bN_0\times[\lambdao, A)$, given by the following:
  \be \label{e:stdistycase1}
    \Pi(n, d\lambda)=\pi(n,\lambda)d\lambda 
    \ee
    whose density is defined as
\be \label{e:stdistycase1.pi}
  \begin{split}
    \pi(0,\lambda)=&\ind_{\{\lambda <\beta\}}\frac{K_0}{\alpha  (\beta -\lambdao)} \left(\frac{\beta -\lambda }{\beta -\lambdao}\right)^{\frac{\beta -\alpha }{\alpha }} e^{\frac{\lambda - \lambdao}{\alpha}},\\
    \pi(n,\lambda)=&\ind_{\{\lambda<\beta\}}\frac{1}{F(\lambda)}\int_{\lambdao}^{\lambda} \frac{F(y)p(n-1,y)}{\alpha(\beta-y)}dy\\
    &+\ind_{\{\lambda\geq \beta\}}\frac{1}{F(\lambda)}\int_{\lambda}^{A} \frac{F(y)p(n-1,y)}{\alpha(y-\beta)}dy, \quad n\geq 1.
  \end{split}
  \ee
  For $n\ge 0$, 
      \be \label{e:p.n.lambda}
    p(n,\lambda) = \int \ind_{\{u+\lambdao<\lambda\}} (\lambda-u) \pi(n,\lambda-u)\,dG_{\ell}(n+1, u),
  \ee
    \be \label{e:F.lambda}
    F(\lambda)=\exp(-\frac{\lambda}{\alpha})|\lambda-\beta|^{\frac{\alpha-\beta}{\alpha}},
  \ee
  and $K_0=\sum_{n=1}^\infty \int_{[\lambdao,A)} \lambda \bar{G}_{\ell}(n, (A-\lambda) \minus) \Pi(n-1, d \lambda)$.
\end{theorem}
  \begin{proof} 
  The detailed proof is provided in Appendix \ref{app:thm:stat.dist.alpha>0.beta>lambda0}.
  \end{proof}

  \begin{theorem}\label{thm:stat.dist.alpha>0.beta=lambda0}
    \addcontentsline{toc}{subsection}{Theorem \thetheorem\ {case \texorpdfstring{$\beta=\lambda_o$}{Theorem β=λₒ}} and \texorpdfstring{$\alpha>0$}{α>0}}
    Let $X=(N,\Lambda,C)$  be a self-exciting PDMP with resetting and $Y=(N,\Lambda)$.  Under Assumption \ref{a:finitereset} and with $\alpha>0\ \&\ \beta=\lambdao$,
    the invariant measure $\Pi$ for $Y$ is given by the following:
    \be \label{e:stdistycase2}
    \begin{split}
      \Pi(0,\{\lambdao\}) &= K_0/\lambdao\\
      \Pi(n, d\lambda)&=\pi(n,\lambda)d\lambda, \quad n\geq 1, \lambda \in (\lambdao, A),
                \end{split}
    \ee
    where the density is given by 
          \be \label{e:stdistycase2.pi}
    \begin{split}
      \pi(1,\lambda)&= {\frac{K_0}{\alpha}\int\ind_{\{\lambda<y+\lambdao<A\}} \frac{F(\lambdao + y)}{yF(\lambda)} \,dG_{\ell}(1, y)}, \mbox{ and} \\
      \pi(n,\lambda)&=
      \frac{1}{\alpha}\int_{\lambda-\lambdao}^{A-\lambdao} \frac{F(\lambdao + y)}{yF(\lambda)} p(n-1,\lambdao + y) \, dy, \; \mbox{ for } n\geq 2
          \end{split}
    \ee
    with $F(\lambda)$ defined in~\eqref{e:F.lambda}, and for $n\ge2$,
\be \label{e:stdistycase2.p}
      p(n-1,\lambda) = \int\ind_{\{u + \lambdao< \lambda\}} (\lambda-u) \pi(n-1,\lambda-u) \,dG_{\ell}(n, u).
    \ee
    The constant $K_0$ is equal to 
    \be \label{e:stdistycase2.K0}
    K_0=\sum_{n=1}^\infty \int_{[\lambdao,A)} \lambda \bar{G}_{\ell}(n, (A-\lambda) \minus) \Pi(n-1, d \lambda) \ .
    \ee
  \end{theorem}
    \begin{proof} 
  The detailed proof is provided in Appendix \ref{app:thm:stat.dist.alpha>0.beta=lambda0}.
  \end{proof}

    \begin{remark} With $n\ge1$, we have that,
    \be
    \pi(n,\lambda) = \pi(n,\lambdao + \ell_n)\left(\frac{\lambda-\lambdao}{\ell_n}\right)^{\frac{\lambdao-\alpha}{\alpha}} e^{-\frac{\lambdao+\ell_n-\lambda}{\alpha}},  \quad \lambda\in (\lambdao, \lambdao + \ell_n).
    \ee
    \end{remark}


  \begin{theorem}\label{thm:stat.dist.alpha=0}  Let $X=(N,\Lambda,C)$  be a self-exciting PDMP with resetting and $Y=(N,\Lambda)$. Under Assumption \ref{a:finitereset} holds and with $\alpha=0$, there exists an invariant measure $\Pi$ for $Y$, given by the following, for $\lambda \in (\lambdao, A)$:
    \be \label{e:stdistycase3}
    \begin{split}
      \Pi(0, d \lambda) &=\frac{K_0}{\lambdao} \delta_{\lambdao}(d\lambda)\\
      \Pi(1, d \lambda) &=\frac{K_0}{\lambda}\,dG_{\ell}(1, \lambda-\lambdao)\\
      \Pi(n, d \lambda) &=\frac{1}{\lambda}  d\overline G_{\ell}(n,\lambda)
    \end{split}
    \ee
    with $K_0=\sum_{n=1}^\infty\int_{[\lambdao,A)}\lambda \bar{G}_{\ell}(n, (A-\lambda) \minus)\Pi(n-1,d\lambda)$ and
    where $\overline \Pi(n,d\lambda) := \lambda \Pi(n,d\lambda)$
    and $\overline G_{\ell}^\star(n,\cdot) :=G_{\ell}(n,\cdot) \star \overline \Pi(n-1,\cdot)$.
  \end{theorem}
      \begin{proof} 
  The detailed proof is provided in Appendix \ref{app:thm:stat.dist.alpha=0} .
  \end{proof}

  \section{An explicit control problem} \label{sec:exponential}
  
  In this section we are going to analyze some specific examples that allow for more explicit calculations. This enables to have some concrete ideas of the dynamics of the process and its stationary distribution, and in few cases to implement and solve, at least numerically, the associated control problem.

We work under the following:
\begin{assumption}\label{ass_ex}
    \begin{itemize}
    \item[i)] $\alpha>0$;
    \item[ii)] $\lambdao=\beta < A <\infty$;
    \item[iii)] The jumps' sizes $\ell(n,Z_n)\sim\Exp(\theta)$ for every $n \ge 1$.
    \end{itemize}
\end{assumption}  

  \begin{theorem}\label{thm:stat.dist.alpha>0.beta=lambda0.exp}
    \addcontentsline{toc}{subsection}{Theorem \thetheorem\ {case \texorpdfstring{$\beta=\lambda_o$}{Theorem β=λₒ}} and \texorpdfstring{$\alpha>0$}{α>0}  and exponential jumps}
    Under Assumption \ref{ass_ex}, the invariant measure $\Pi$ for $\Lambda$ is given by the following:
    \be \label{e:stdistycase.example}
    \Pi(d\lambda) = \frac{K_0}{\beta} \delta_{\beta}(d\lambda) + \pi(\lambda) d\lambda
    \ee
    where the density is given by 
    \be \label{e:stdistycase.example.pi}
    \pi(\lambda)
    = \frac{K_0}{\beta} \frac{\beta\theta}{\alpha(\lambda-\beta)} \int_\lambda^A \left(\frac{u-\beta}{\lambda-\beta}\right)^{-\frac{\beta}{\alpha}} e^{-(\frac{1}{\alpha}-\theta)(u-\lambda)} du \ ,
    \ee
    with $K_0/\beta = \frac{\alpha}{\alpha+ \beta \theta K_3}$ and 
  \be\label{eq:K3}
  K_3 =
  \begin{cases}
    -{\Kummer}^{(1,0,0)}\left(0,\frac{\beta }{\alpha },\frac{(\alpha  \theta -1) (A-\beta )}{\alpha }\right), & \alpha \not= 1/\theta \\
    \frac{\alpha}{\beta}  \left( A -\beta\right) & \alpha = 1/\theta
  \end{cases}
  \ee
  where ${\Kummer}^{(1,0,0)}$ is the Kummer’s confluent hypergeometric function.
  \end{theorem}
  \begin{proof} 
  The detailed proof is provided in Appendix \ref{app:thm:stat.dist.alpha>0.beta=lambda0.exp}.
  \end{proof}

  \begin{figure}
    \includegraphics{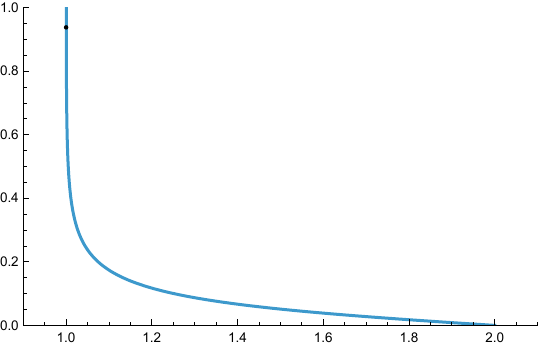}
    \caption{Plot of the density $\pi(\lambda)$ given in~\eqref{e:stdistycase.example.pi} with parameters $\beta=1$, $A=2$, $\theta=0.1$ and $\alpha=1.1$.\label{fig.explicit.pi} }
  \end{figure}
In the Figure~\ref{fig.explicit.pi} it is shown an explicit plot of the stationary distribution for $\Lambda$ for some explicit values of the parameters.

Let us now return to our control problem~\eqref{eq:problem}, and assume that the cost of the technology is $\eta(A)=c/A$ for some $c>0$.
    Thanks to Corollary~\ref{c:Climit}, 
    \begin{equation*}
     J(A) = \eta(A) + \mu_Z \lim_{t \rightarrow +\infty}  \frac{\int_0^t \Lambda_s ds}{t} .
    \end{equation*}
    Now, since $(N,\Lambda)$ is ergodic, it follows from, e.g., Th\'eor\`eme on p.30 in Az\'ema, Duflo and Revuz \cite{ADR} that
    \be
    \lim_{t \rightarrow +\infty} \frac{1}{t}\int_0^t \Lambda_s ds = \E_A [\Lambda_\infty] := \sum_{n=0}^\infty\int \lambda \Pi_A(n,d\lambda),
    \ee
    where $\Pi_A$ is the invariant distribution when the intensity is reset at $T_A$. Thus, our problem reads:
    \begin{equation}
      V = \inf_{A \in \R_+} \left( \frac{c}{A} + \mu_Z \E_A [\Lambda_\infty] \right) \ .
    \end{equation}

    \begin{proposition}Under Assumption~\ref{ass_ex}, with  $\alpha = 1/\theta$ and $\eta(A) = c/A$, the value function is equal to 
    \be\label{explicit.V}
    V = \frac{ \alpha \beta (2 \alpha^2 + A_*^2 +2 \alpha A_* - \beta^2)}{2 (\alpha + \beta)(\alpha + A_* - \beta)} + \frac{c}{A_*}
    \ee
    where
    \be\label{eq:optA.alpha=1/theta}
      A_* =\frac{\alpha ^2 \beta ^2+4 c^2 \left(\alpha ^2-\beta ^2\right)^2-4 \alpha  \beta  c \left(\alpha ^2-\beta ^2\right)-4 c \rho_1 \left(\alpha ^2-\beta ^2\right)+\rho_1^2-\alpha  \beta  \rho_1}{6 c (\alpha +\beta ) \rho_1}
    \ee
    where
    \begin{align*}
      \rho_1^3 =&-\alpha ^3 \beta ^3+8 c^3 \left(\alpha ^2-\beta ^2\right)^3+6 \alpha  \beta  c^2 (\alpha +\beta )^2 \left(7 \alpha ^2+4 \alpha  \beta -2 \beta ^2\right)\\
      &+6 \alpha ^2 \beta ^2 c (\alpha -\beta ) (\alpha +\beta )
      +6 \sqrt{3} \sqrt{\alpha ^3 \beta  c^2 (\alpha +\beta )^2 \rho_2}\\
      \rho_2 =& -\alpha ^3 \beta ^3+8 c^3 \left(\alpha ^2-\beta ^2\right)^3+3 \alpha  \beta  c^2 (\alpha +\beta )^2 (\alpha +2 \beta ) (5 \alpha -2 \beta )\\
      &+6 \alpha ^2 \beta ^2 c (\alpha -\beta ) (\alpha +\beta ).
    \end{align*}
    \end{proposition}

    \begin{proof}
    Since $\Pi_A(d \lambda)= \frac{K_0}{\beta} \delta_{\beta}(d \lambda) + \pi(\lambda) d\lambda$, we first exploit  Eqs. \eqref{eq:K0} and \eqref{eq:K3} to find
    $$
    K_0 = \frac{\alpha \beta}{ \alpha  + \beta \theta K_3} = \frac{\alpha \beta}{ \alpha  + A - \beta }
    $$
    and then, recalling Eq. \eqref{eq:pi.explicit.final}, we obtain
    \begin{align*}
      \int_{\beta}^A \frac{\lambda}{\lambda  -\beta}\left(  \int_\lambda ^A \left(\frac{u-\beta}{\lambda -\beta}\right)^{-\frac{\beta}{\alpha}}  du \right) d\lambda
      &= \frac{\alpha (A-\beta)(A + 2\alpha +\beta)}{2(\alpha + \beta)}
    \end{align*}
    and so
    \be
    \begin{split}
      \E_A[\Lambda_\infty] &= K_0 + K_0 \frac{\theta }{\alpha}\int_{\beta}^A\frac{\lambda}{\lambda -\beta}  \left(\int_\lambda^A \left(\frac{u-\beta}{\lambda-\beta}\right)^{-\frac{\beta}{\alpha}}  du \right)d\lambda\\
      &= \frac{\alpha \beta}{ \alpha  + A - \beta }
      + \frac{\beta}{ \alpha  + A - \beta } \frac{(A-\beta)(A + 2\alpha +\beta)}{2(\alpha + \beta)}\\
      &= \frac{\beta (2 \alpha^2 + A^2 +2 \alpha A - \beta^2)}{2 (\alpha + \beta)(\alpha + A - \beta)}.
    \end{split}
    \ee
    It follows that the cost function is given by
    $$
    J(A) = \alpha \E_A [\Lambda_\infty] + \frac{c}{A} = \frac{ \alpha \beta (2 \alpha^2 + A^2 +2 \alpha A - \beta^2)}{2 (\alpha + \beta)(\alpha + A - \beta)} + \frac{c}{A}
    $$
    and it is minimized 
    at $A_*$ given in~\eqref{eq:optA.alpha=1/theta}.
    \end{proof}

    \begin{figure}
      \includegraphics{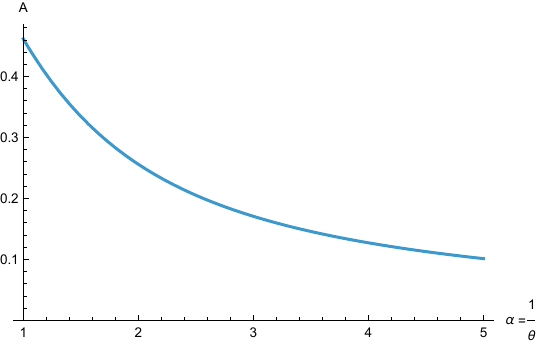}
      \caption{Optimal threshold $A_*$ as a function of $\alpha=1/\theta$,  given in~\eqref{eq:optA.alpha=1/theta} with parameters $\beta=1$, $\eta(A)=c/A$, $c=2$.\label{fig.explicit.opt.A}}
    \end{figure}
    In the Figure~\ref{fig.explicit.opt.A} it is shown an explicit plot of the optimal threshold $A_*$ as a function of $\alpha=1/\theta$. 
    
    \bibliographystyle{abbrvnat}
    \bibliography{main}
    \newpage

\begin{appendix}
\section{Proofs}
\subsection{Proof of Theorem \ref{thm:hilleyosida}\label{app:thm:hilleyosida}}
The result will follow from the Hille-Yosida theorem (see, e.g., Theorem 1.2.12 in Ethier and Kurtz \cite{ethier2009markov}). 

It is clear that the domain of $Q^\alpha$ is dense in $\widehat{C}({S})$.
To see that it is dissipative let $f$ be in its domain. Note that $f$ is the restriction of some function in $\widehat{C}(\bar{S})$. We will denote this function by $f$ as well. Since it vanishes at infinity, there exists a point $x^\star \in \bar{S}$ such that $\|f\|=f(x^\star)$. Without loss of generality we can suppose that $f(x^\star)\geq 0$.
If $x^\star$ is in the interior of $S$, $f_\lambda(x^\star)=0$, which in turn yields that $Q^\alpha f(x^\star)\leq 0$.
If $x^\star=(n^\star, A, z^\star)$, then it is clear that $f_\lambda(x^\star)\geq 0$.
Then, since~$\beta<A$, we deduce $Q^\alpha f(x^\star):=\lim_{x\to x^\star} Q^\alpha f(x) \leq 0$.
Similar argument proves that $Q^\alpha f(x^\star)\leq 0$ in the other boundary case as well.
Next, consider $r>0$ and note that
\[
\|rf-Q^\alpha f\|\geq rf(x^\star)-Q^\alpha f(x^\star)\geq r\|f\|,
\]
which demonstrates that $Q^\alpha$ is dissipative.

Now, it remains to show that the range of $r I - Q^\alpha$ is dense in $\widehat{C}(S)$ for some $r>0$, where $I$ is the identity operator. We claim that for any $g \in C_1(\bar{S})$ there exists a solution to the the equation
\be \label{e:fullrange}
rf -Q^\alpha f =g.
\ee
We shall provide the proof for $\alpha>0$. The case of $\alpha=0$ follows from easier arguments and is left to the reader.

To this end, for a given $f \in \widehat{C}(\bar{S})$, let $Tf(n,z)$ be the solution of the following ODE with a suitable boundary condition for fixed $n$ and $z$:
\be \label{e:defT}
(r+\lambda)u -\alpha(\beta-\lambda)u_\lambda
= g + \lambda \, Rf(n,\lambda,z),
\ee
where 
\[\begin{split}
Rf(n,\lambda,z)
:= &\int\ind_{\{\ell(n+1,y)+\lambda < A\}}f(n+1,\lambda +\ell(n+1,y), z+y) \,dG(y)\\
&+ \int\ind_{\{\ell(n+1,y)+\lambda \ge A\}}f(0,\lambdao, z+y) \,dG(y) \ .
\end{split}
\]
By finiteness and continuity at $\lambda=\beta$, we have $Tf\in \widehat{C}(\bar{S})$ 
given by 
\[
Tf(n,\lambda,z) = 
e^{-F(\lambda)} \int_\beta^\lambda e^{F(y)}\frac{g(y)+y Rf(n,y,z)}{\alpha (y-\beta)}dy ,
\]
with $Tf(n,\beta,z):=0$ and 
\[
F(\lambda) := \frac{r+\beta}{\alpha}\log |\lambda-\beta|+\frac{\lambda}{\alpha}.
\]

We claim that this operator is a contraction on $\widehat{C}(\bar{S})$ for a large enough $r$.  Indeed, for $\{f_1,f_2\}\subset  \widehat{C}(\bar{S})$, and $\lambda \in [\lambdao,A]$
\[\begin{split}
|Tf_1(n,\lambda,z)- Tf_2(n,\lambda,z)|
&\leq 
\int_\beta^\lambda \frac{y e^{F(y)-F(\lambda)}}{\alpha (y-\beta)} \|R f_1- R f_2\| dy\\
&\leq 
\left(\int_\beta^\lambda \frac{y e^{F(y)-F(\lambda)}}{\alpha (y-\beta)} dy \right)  \|f_1-f_2\|.
\end{split}
\]

Note that, since 
for $0<y<|\lambda-\beta|$ and $r>1$, we have 
\[
\frac{1 }{y}e^{F(y+\beta)-F(\lambda)}
=  \frac{1}{|\lambda-\beta|} \left(\frac{y}{|\lambda-\beta|}\right)^{\frac{r-1+\beta}{\alpha}} e^{\frac{y}{\alpha} + \frac{\beta-\lambda}{\alpha}}
\leq \frac{e^{ \frac{2|\lambda-\beta|}{\alpha}}}{|\lambda-\beta|}  \left(\frac{y}{|\lambda-\beta|}\right)^{\frac{r-1+\beta}{\alpha}}
\]
we have for $\lambda\in[\lambdao,A]$,
\[\begin{split}
\int_\beta^\lambda \frac{y e^{F(y)-F(\lambda)} }{\alpha (y-\beta)}dy
&=\int_0^{|\lambda-\beta|} \frac{(y+\beta) e^{F(y+\beta)-F(\lambda)} }{\alpha y}dy\\
&\leq \frac{\lambda\vee\beta}{\alpha} e^{ \frac{2|\lambda-\beta|}{\alpha}} 
\int_0^{|\lambda-\beta|}  \left(\frac{y}{|\lambda-\beta|}\right)^{\frac{r-1+\beta}{\alpha}} d\frac{y}{|\lambda-\beta|} 
=\frac{\lambda\vee\beta}{r+\beta} e^{\frac{2|\lambda-\beta|}{\alpha}  \, 
},
\end{split}\]
where $a \vee b = \max\{a,b\}.$
Therefore, for large enough $r$,
\[
\|Tf_1-Tf_2\|\leq \frac{1}{2}\|f_1-f_2\|.
\]
That is, $T$ is a contraction. Hence, it admits a fixed point, which demonstrates that \eqref{e:fullrange} has a solution. This completes the proof.

\subsection{Proof of Theorem \ref{thm:stat.dist.alpha>0.beta>lambda0} \label{app:thm:stat.dist.alpha>0.beta>lambda0}}
  For a test function $f$, by definition $\Pi$ must satisfy
  \[
    \sum_{n=0}^\infty \int_{[\lambdao,A)}Q^\alpha f(n,\lambda) \Pi(n,d\lambda)=0.
  \]
  Thus, using Equation \eqref{e:QalphaY}, we find
  \be\label{e:Qeqforsd}
  \begin{split}
    \sum_{n=0}^\infty \int_{[\lambdao,A)} &\lambda f(n,\lambda) \Pi(n,d\lambda) \\
    =&f(0,\lambdao)K_0 + \sum_{n=0}^\infty \int_{[\lambdao,A)}\alpha (\beta-\lambda) f_\lambda(n,\lambda)\Pi(n,d\lambda)\\
    &+\sum_{n=1}^\infty \int_{[\lambdao,A)}\int\ind_{\{u+\lambda<A\}}\lambda f(n,u+\lambda)\,dG_{\ell}(n, u)\Pi(n-1,d\lambda),
  \end{split}
  \ee
  where $K_0=\sum_{n=1}^\infty\int_{[\lambdao,A)}\lambda \bar{G}_{\ell}(n, (A-\lambda) \minus)\Pi(n-1,d\lambda)$.


  We start by choosing a test function $f$ such that $f(n,\lambda)=0$, for $n>0$, and assuming
  $$
  \Pi(0,d\lambda)=\Pi(0,\{\lambdao\}) \delta_{\lambdao}(d\lambda) + \pi(0,\lambda)d\lambda.
  $$
  Moreover, $\pi(0,\lambda)$ is of finite variation and therefore we write its decomposition as
  $$
  d\pi(0,\lambda)=\pi'(0,\lambda)d\lambda + \widehat \pi'(0,d\lambda)
  $$
  where the (signed) measure $\widehat \pi'(0,d\lambda)$ is singular with respect to the Lebesgue measure and with support in $(\lambda_0,A)$.

  Under these assumptions and applying integration by parts' formula,  \eqref{e:Qeqforsd} reduces to
  \be\label{e:Qeqforsd.n=0}
  \begin{split}
    0 =
    & f(0,\lambdao)\left(
      \lambdao\Pi(0,\{\lambdao\}) - K_0 + \alpha (\beta-\lambdao) \pi(0,\lambdao)
    \right)\\
    &-f_\lambda(0,\lambdao) \alpha (\beta-\lambdao) \Pi(0,\{\lambdao\})
    - f(0,A) \alpha (\beta-A) \pi(0,A) \\
    &+ \int_{(\lambdao,A)} f(0,\lambda) \alpha (\beta-\lambda)  \widehat \pi'(0,d\lambda)\\
    &+\left(\int_\lambdao^\beta +\int_\beta^A \right) f(0,\lambda) \left((\lambda - \alpha) \pi(0,\lambda) + \alpha (\beta-\lambda)  \pi'(0,\lambda) \right) d\lambda .
  \end{split}
  \ee
  Since $f(0,\lambda)$ is arbitrarily chosen in a set of functions that separates measures, we get that the support of the measure $\widehat\pi'(0,d\lambda)$ is $\{\beta\}$, $\Pi(0,\{\lambdao\}) = 0$
  and $\pi(0, \lambda)$ solves the following ODE in $[\lambdao,A]\setminus\{\beta\}$,
  \be
  (\lambda - \alpha) \pi(0,\lambda) + \alpha (\beta-\lambda)  \pi'(0,\lambda) = 0
  \ee
  with boundary conditions
  \be\label{e:pi0ode}
  \begin{split}
    \pi(0,\lambdao) &= K_0/(\alpha (\beta-\lambdao)) \\
    \pi(0,A) &= 0
  \end{split}
  \ee

  It follows that, for $\lambda\in[\lambdao,\beta)$,
  \be \label{e:pi0case1}
  \pi(0,\lambda)=\frac{K_0}{\alpha  (\beta -\lambdao)} \left(\frac{\beta -\lambda }{\beta -\lambdao}\right)^{\frac{\beta -\alpha }{\alpha }} e^{\frac{\lambda - \lambdao}{\alpha}} ,
  \ee
  and $\pi(0,\lambda)=0$ outside that interval. In particular, $\widehat\pi'(0,\{\beta\})=-\pi(0,\beta-)$.
  
  We again look at \eqref{e:Qeqforsd}, now for $n>0$, assuming that $f(m,\lambda)=0$ for $m\not=n$ and the knowledge of $\Pi(n-1,\lambda)$. Since $\Pi(0,\lambda)$ does not charge any mass to any point, we do not expect any point mass of $\Pi(n,\lambda)$, so we assume
  $$
  d\Pi(n,\lambda) = \pi(n,\lambda)d\lambda, \quad d\pi(n,\lambda)=\pi'(n,\lambda)d\lambda + \widehat \pi'(n,d\lambda)
  $$
  where the (signed) measure $\widehat \pi'(n,d\lambda)$ is singular with respect to the Lebesgue measure and with support inside $(\lambda_0, A)$.

  Under these assumptions and applying integration by parts' formula,  \eqref{e:Qeqforsd} reduces to
  \be\label{e:Qeqforsd.n}
  \begin{split}
    0 =
    & f(n,\lambdao)\left(\alpha (\beta-\lambdao) \pi(n,\lambdao)
    \right)
    - f(n,A) \alpha (\beta-A)  \pi(n,A) \\
    &+ \int_{(\lambda^o,A)} f(n,\lambda) \alpha (\beta-\lambda)  \widehat \pi'(n,d\lambda)\\
    &+\left(\int_\lambdao^\beta +\int_\beta^A \right)  f(n,\lambda) \left((\lambda - \alpha) \pi(n,\lambda) + \alpha (\beta-\lambda)  \pi'(n,\lambda) \right) d\lambda \\
    &-\int_\lambdao^A f(n,\lambda)  p(n-1,\lambda) d\lambda
  \end{split}
  \ee
  where
  \be \label{e:monster_p}
  p(n-1,\lambda) = \int\ind_{\{u+\lambdao<\lambda\}} (\lambda-u) \pi(n-1,\lambda-u)\,dG_{\ell}(n, u).
  \ee

  By analyzing the expression \eqref{e:Qeqforsd.n}, we infer that $\pi(n,\lambdao)=\pi(n,A)=0$ and that the support of the measure $\widehat\pi'(n,d\lambda)$ is $\{\beta\}$, again because $f(0,\lambda)$ is arbitrarily chosen in a set of functions that separates measures.

  Moreover,the density function $\pi(n,\lambda)$, satisfies the following ODE, in the regions $[\lambdao, \beta)$ and $[\beta, A]$
  \be\label{diff.eq.pi.n.1}
  \begin{split}
    (\lambda - \alpha) \pi(n,\lambda) + \alpha (\beta-\lambda)  \pi'(n,\lambda) &= p(n-1,\lambda)   \\
    \pi(n,\lambdao)=\pi(n,A)&=0.
  \end{split}
  \ee
  As before, we deduce $\widehat\pi'(n,\{\beta\})=\pi(n,\beta)-\pi(n,\beta \minus)$.

  To solve the ODE with the given boundary behavior, use  the integrating factor $F$ defined by \eqref{e:F.lambda}, which 
  has the property that, for $\lambda\not=\beta$, $F'(\lambda)/F(\lambda)=(\alpha-\lambda)/(\alpha(\lambda-\beta))$.
  Then, computing the derivative of $F(\lambda)\pi(n,\lambda)$ and using \eqref{diff.eq.pi.n.1}, we get
  \[
    \partial_\lambda (F(\lambda)\pi(n,\lambda)) = \frac{F(\lambda)p(n-1,\lambda)}{\alpha(\beta-\lambda)}.
  \]
  By integrating on~$[\lambdao,\lambda)$ and using the boundary condition at $\lambdao$, we get on~$(\lambdao,\beta)$
  \be \label{e:pincase1a}
  \pi(n,\lambda)=\frac{1}{F(\lambda)}\int_{\lambdao}^{\lambda} \frac{F(y)p(n-1,y)}{\alpha(\beta-y)}dy,
  \ee
  that is null on $[\lambdao,(\lambdao+\Ell_n) \wedge \beta]$.
  Using the boundary condition at $A$, we also get on $[\beta,A)$
  \be \label{e:pincase1b}
  \pi(n,\lambda) = \frac{1}{F(\lambda)}\int_{\lambda}^{A} \frac{F(y)p(n-1,y)}{\alpha(y-\beta)}dy.
  \ee
  We shall next show that $\pi(n,\cdot)$ is integrable on $(\lambda^o, A)$ that allows to construct the stationary distribution after re-normalizing by the constant $K_0$, (note that the normalizing constant always exists  because $\pi(0,\cdot)$ does not vanish everywhere, so the sum $\sum_{n=0}^\infty \pi(n,\cdot)$ is not identically null). Since it is continuous except possibly at $\beta$, the integrability will follow from $\lim_{\lambda \to \beta}(\beta-\lambda)\pi(n,\lambda)=0$.  To this end, consider
  \[
    \lim_{\lambda\uparrow \beta}\frac{\beta-\lambda}{F(\lambda)}=\exp(\beta/\alpha)\lim_{\lambda\uparrow \beta}(\beta-\lambda)^{\frac{\beta}{\alpha}}=0.
  \]
  The above  implies that $\lim_{\lambda\to \beta} (\beta-\lambda) \pi(n,\lambda)=0$ provided $\int_{\lambdao}^{\lambda} \frac{F(y)p(n-1,y)}{\alpha(\beta-y)}dy<\infty$. If it is not finite, an application of de l'H\^opital's rule yields that
  \begin{align*}
    \lim_{\lambda\uparrow  \beta} (\beta-\lambda) \pi(n,\lambda)&=\lim_{\lambda\uparrow  \beta}\Big(\frac{d}{d\lambda}\log (F(\lambda)/(\beta-\lambda)\Big)^{-1}\frac{p(n-1,\lambda)}{\alpha}  \\
      &=\frac{1}{\beta}\lim_{\lambda\uparrow  \beta}(\beta-\lambda)p(n-1,\lambda).
    \end{align*}
    On the other hand, it follows by an induction argument that $\lim_{\lambda\uparrow  \beta}(\beta-\lambda)p(n-1,\lambda)=0$ since $\lim_{\lambda \uparrow  \beta}(\beta-\lambda)p(0,\lambda)=0$ using the explicit form of $\pi(0,\cdot)$. Thus, $\lim_{\lambda \uparrow \beta} (\beta-\lambda) \pi(n,\lambda)=0$ for all $n\geq 0$.
    
\subsection{Proof of Theorem \ref{thm:stat.dist.alpha>0.beta=lambda0} \label{app:thm:stat.dist.alpha>0.beta=lambda0}}
    As for the proof of Theorem~\ref {thm:stat.dist.alpha>0.beta>lambda0}, we look at the equality in~\eqref{e:Qeqforsd} by choosing a test function $f$ such that $f(n,\lambda)=0$, for $n>0$, and assuming $\Pi(0,d\lambda)=\Pi(0,\{\lambdao\}) \delta_{\lambdao}(\lambda) + \pi(0,\lambda)d\lambda$. We assume $\pi(0,\lambda)$ is of finite variation and therefore we write its decomposition as $d\pi(0,\lambda)=\pi'(0,\lambda)d\lambda + \widehat \pi'(0,d\lambda)$ where the (signed) measure $\widehat \pi'(0,d\lambda)$ is singular with respect to the Lebesgue measure.

    Under these assumptions \eqref{e:Qeqforsd} reduces to~\eqref{e:Qeqforsd.n=0}, that, with $\beta=\lambdao$, simplifies to
    \be\label{e:Qeqforsd.n=0.beta=lamba0}
    \begin{split}
      0 =
      & f(0,\lambdao)\left(\lambdao \Pi(0,\{\lambdao\}) - K_0 \right) + f(0,A) \alpha (A-\lambdao) \pi(0,A) \\
      &- \int_{(\lambdao,A)} f(0,\lambda) \alpha (\lambda-\lambdao)  \widehat \pi'(0,d\lambda)\\
      &+\int_\lambdao^A f(0,\lambda) \left((\lambda - \alpha) \pi(0,\lambda) - \alpha (\lambda-\lambdao)  \pi'(0,\lambda) \right) d\lambda .
    \end{split}
    \ee
    By the arbitrariness of $f(0,\cdot)$, we get $\Pi(0,\{\lambdao\}) = K_0/\lambdao$, while $\pi(0,\lambda)$ is null on the whole interval $(\lambdao,A)$.


    Now we assume that $f(n,\lambda)=0$ for $n\not=1$. We assume that $d\Pi(1,\lambda) = \pi(1,\lambda)d\lambda$ and $d\pi(1,\lambda)=\pi'(1,\lambda)d\lambda + \widehat \pi'(1,d\lambda)$ where the (signed) measure $\widehat \pi'(1,d\lambda)$ is singular with respect to the Lebesgue measure, and  $\widehat \pi'(1,\{\lambdao\})=0$ by convention.
    We also assume
    \begin{equation}\label{G.decomp}
    dG_{\ell}(1, \lambda-\lambdao) = G'_{\ell}(1, \lambda-\lambdao) d\lambda + \widehat G'_{\ell}(1, d(\lambda-\lambdao))
    \end{equation}
    with $\widehat G'_{\ell}$ singular with respect to the Lebesgue measure.

    So, applying once more integration by parts' formula as done to obtain Equation \eqref{e:Qeqforsd.n}, and using that $\Pi(0,\cdot)$ is non-null only at $\{\lambdao \}$, with $\Pi(0,\{\lambdao\}) = K_0/\lambdao$, Equation \eqref{e:Qeqforsd} reduces to:
    \be\label{e:Qeqforsd.n=1}
    \begin{split}
      0=& \int_{[\lambdao, \lambdao + \ell_1)} f(1,\lambda) \left((\lambda-\alpha) \pi(1,\lambda) - \alpha (\lambda-\lambdao) \pi'(1,\lambda)\right)d\lambda \\
      &-\int_ {(\lambdao, \lambdao + \ell_1)} f(1,\lambda) \alpha (\lambda-\lambdao) \widehat \pi'(1,d\lambda)\\
      &-\int_{[\lambdao + \ell_1,A)} f(1,\lambda) \left(-(\lambda-\alpha) \pi(1,\lambda) + \alpha (\lambda-\lambdao) \pi'(1,\lambda) + K_0 G'_{\ell}(1, \lambda-\lambdao)\right) d\lambda \\
      &-\int_{[\lambdao + \ell_1, A)} f(1,\lambda) \left(\alpha (\lambda-\lambdao) \widehat \pi'(1,d\lambda) +K_0 \widehat G'_{\ell}(1, d(\lambda-\lambdao))\right) \\
      &-\alpha (\lambdao-A) f(1,A) \pi(1,A) + f(1,\lambdao)\alpha\lim_{\lambda \to \lambdao}(\lambda-\lambdao)\pi(1,\lambda).
    \end{split}
    \ee

    It follows that $\pi(1,A)=0$, $\lim_{\lambda \to \lambdao}(\lambda-\lambdao)\pi(1,\lambda)=0$ and $\widehat \pi'(1,d\lambda)$ has support in $[\lambdao + \ell_1, A)$ where
    $$
    \widehat \pi'(1,d\lambda)= -\frac{K_0}{\alpha(\lambda-\lambdao)} \widehat G'_{\ell}(1, d(\lambda-\lambdao)).
    $$
    We now focus separately on the two sub-intervals:
    \begin{itemize}
      \item On $[\lambdao + \ell_1,A)$:
        \be \label{e:odepi1start}
        (\lambda-\alpha) \pi(1,\lambda) - \alpha (\lambda - \lambdao) \pi'(1,\lambda) = K_0 G'_{\ell}(1, \lambda-\lambdao)
        \ee
        and, since $\pi(1,\lambda)=  -\int_\lambda^A \pi'(1,u) du - \int_{(\lambda,A)} \widehat\pi'(1,du)$,
        we have
        \be
        \begin{split}
          &  -(\lambda-\alpha)\int_\lambda^A \pi'(1,u) du - \alpha (\lambda - \lambdao) \pi'(1,\lambda)
          = p(0,\lambda)\label{e:odepi1}
        \end{split}
        \ee
        where
        \begin{equation*}
        \begin{split}
          p(0,\lambda)
          =K_0 \Big(G'_{\ell}(1, \lambda-\lambdao) & - \int_{(\lambda,A)}\frac{(\lambda- \alpha)}{\alpha (u-\lambdao)} \widehat G'_{\ell}(1, d(u-\lambdao))\Big)\\
          \end{split}
        \end{equation*}
        Define  $\pi_0(1,\lambda):=-\int_\lambda^A \pi'(1,u) du$, so that \eqref{e:odepi1} reads
        \be
        (\lambda-\alpha)\pi_0(1,\lambda)- \alpha (\lambda - \lambdao) \pi'_0(1,\lambda) = p(0,\lambda).
        \ee
        Using the boundary condition at $A$ we find for $\lambda \in [\lambdao + \ell_1,A)$
        \[
          \pi_0(1,\lambda)=\frac{1}{F(\lambda)}\int_{\lambda}^{A} \frac{F(y)p(0,y)}{\alpha(y-\lambdao)}dy,
        \]
        where $F$ is defined in \eqref{e:F.lambda}. Consequently, for $\lambda \in [\lambdao + \ell_1,A)$,
        \begin{equation*}
        \begin{split}
          \pi(1,\lambda) 
          &= \frac{1}{F(\lambda)}\int_{\lambda}^{A} \frac{F(y)p(0,y)}{\alpha(y-\lambdao)}dy + \int_{(\lambda,A)} \frac{K_0}{\alpha(u-\lambdao)} \widehat G'_{\ell}(1, d(u-\lambdao)) \\
          &= \frac{K_0}{F(\lambda)}\int_{\lambda}^{A} \frac{F(y) }{\alpha(y-\lambdao)} G'_{\ell}(1, y-\lambdao) dy \\
           &-\frac{K_0}{F(\lambda)} \int_{(\lambda,A)} \frac{1}{\alpha (u-\lambdao)} \left(\int_{\lambda}^{u} F(y) \frac{(y- \alpha)}{\alpha(y-\lambdao)}  dy \right)\widehat G'_{\ell}(1, d(u-\lambdao))  \\
          &+K_0\int_{(\lambda,A)} \frac{1}{\alpha(u-\lambdao)} \widehat G'_{\ell}(1, d(u-\lambdao))
         \end{split}
        \end{equation*}
        and now recalling that being $\lambda\not=\lambdao$, $F'(y)/F(y)=(\alpha-y)/(\alpha(y-\lambdao))$, we find 
        \begin{equation*}
        \begin{split}
          \pi(1,\lambda) 
          &= \frac{K_0}{F(\lambda)}\int_{\lambda}^{A} \frac{F(u)}{\alpha(u-\lambdao)} G'_{\ell}(1, u-\lambdao) du \\       
          &+ \frac{K_0}{F(\lambda)} \int_{(\lambda,A)} \Big(F(u)-F(\lambda)\Big)\frac{1}{\alpha (u-\lambdao)} \widehat G'_{\ell}(1, d(u-\lambdao))  \\
        &+K_0\int_{(\lambda,A)} \frac{1}{\alpha(u-\lambdao)} \widehat G'_{\ell}(1, d(u-\lambdao))\\
          &=  \frac{K_0}{F(\lambda)}\int_{\lambda}^{A} \frac{F(u)}{\alpha(u-\lambdao)} \Big(G'_{\ell}(1, u-\lambdao) du + \widehat G'_{\ell}(1, d(u-\lambdao)) \Big)\\
          &
        =\frac{K_0}{\alpha}\int_{\lambda-\lambdao}^{A-\lambdao} \frac{F(\lambdao + y)}{yF(\lambda)}\,dG_{\ell}(1, y).\\
          \end{split}
        \end{equation*}

      \item On $(\lambdao,\ell_1+\lambdao)$: we have that
        \be
        (\lambda-\alpha) \pi(1,\lambda) - \alpha (\lambda-\lambdao) \pi'(1,\lambda) = 0
        \ee
        with the  boundary condition $\pi(1,\ell_1+\lambdao\minus) = \pi(1,\ell_1+\lambdao)-\frac{K_0}{\alpha \ell_1}  \Delta G_{\ell}(1, \ell_1)=\pi(1,\ell_1+\lambdao)-\frac{K_0}{\alpha \ell_1}  G_{\ell}(1, \ell_1)$ since the support of $G_\ell(1,\cdot)$ is in $[\ell_1, \infty)$ and $  G_{\ell}(1, \ell_1\minus)=0$.

        It follows that, for $\lambda\in(\lambdao,\ell_1+\lambdao)$,
        \be
        \pi(1,\lambda)=\pi(1,\ell_1+\lambdao)\left(\frac{\lambda-\lambdao}{\ell_1}\right)^{\frac{\lambdao-\alpha}{\alpha}} e^{-\frac{\lambdao+\ell_1-\lambda}{\alpha}} .
        \ee
        Note that $\lim_{\lambda \to \lambdao}(\lambda-\lambdao)\pi(1,\lambda)=0$ as required before in this proof.
    \end{itemize}


    Next, to handle the case  for $n>1$, we assume that $f(m,\lambda)=0$ for $m\not=n$ and thereby deduce that  $\Pi(n-1,d\lambda)$ is absolutely continuous  with respect to the Lebesgue measure.
    Therefore, we assume $d\Pi(n,\lambda) = \pi(n,\lambda)d\lambda$ and consider $d\pi(n,\lambda)=\pi'(n,\lambda)d\lambda + \widehat \pi'(n,d\lambda)$ where the (signed) measure $\widehat \pi'(n,d\lambda)$ is singular with respect to the Lebesgue measure and  $\widehat \pi'(n,\{\lambdao\})=0$ by convention as before.

    Applying the integration by parts formula, Equation \eqref{e:Qeqforsd} now reduces to
      \be\label{e:Qeqforsd.n.beta=lambdao}
      \begin{split}
        0 =
        &\int_{(\ell_n,A)} f(n,\lambda) \alpha (\lambdao-\lambda)  \widehat \pi'(n,d\lambda)\\
        &+\int_\lambdao^A f(n,\lambda) \left((\lambda - \alpha) \pi(n,\lambda) - \alpha (\lambda-\lambdao)  \pi'(n,\lambda) - p(n-1,\lambda)\right) d\lambda \\
        &-\alpha (\lambdao-A) f(n,A) \pi(n,A\minus) + f(n,\lambdao)\alpha\lim_{\lambda \to \lambdao}(\lambda-\lambdao)\pi(n,\lambda).
      \end{split}
      \ee
    where
    \be\label{def:p(n-1,lambda).beta=lambdao}
    p(n-1,\lambda) := \int \ind_{\{u+\lambdao<\lambda\}} (\lambda-u) \pi(n-1,\lambda-u)\,dG_{\ell}(n, u).
    \ee

    By analyzing the expression \eqref {e:Qeqforsd.n.beta=lambdao}, we infer that $\pi(n,A\minus)=\lim_{\lambda \to \lambdao}(\lambda-\lambdao)\pi(n,\lambda)=0$ and  the measure $\widehat\pi'(n,d\lambda)$ is null.

    The density function $\pi(n,\lambda)$, satisfies the following ODE on $(\lambdao, A)$
      \be\label{diff.eq.pi.n.2}
      (\lambda - \alpha) \pi(n,\lambda) - \alpha (\lambda-\lambdao)  \pi'(n,\lambda) = p(n-1,\lambda)
      \ee
    with boundary condition $\pi(n,A)=0$.
    Reasoning as before, on $(\lambdao,A)$, we have
    \[
      \pi(n,\lambda)=\frac{1}{F(\lambda)}\int_{\lambda}^{A} \frac{F(y)p(n-1,y)}{\alpha(y-\lambdao)}dy.
    \]
\subsection{Proof of Theorem \ref{thm:stat.dist.alpha=0}\label{app:thm:stat.dist.alpha=0}}
    First of all, when $\alpha=0$ the generator of $Y$ reads:
    \be
    \begin{split}
      Q^0 f(n, \lambda)&= -\lambda f(n, \lambda) 
      +\lambda \int\ind_{\{u+\lambda<A\}}f(n+1,u+\lambda)\,dG_{\ell}(n+1, u)\\
      &+\lambda f(0,\lambdao) \bar{G}_{\ell}(n+1, (A-\lambda) \minus),
    \end{split}
    \ee
    so that, using Equation \eqref{e:QalphaY}, we find that the stationary distribution satisfies
    \be\label{eq:Stat_d_alpha0}
    \begin{split}
      \sum_{n=0}^\infty \int_{[\lambdao,A)} &\lambda f(n,\lambda) \Pi(n,d\lambda) \\
      =&f(0,\lambdao)K_0 +\sum_{n=1}^\infty \int_{[\lambdao,A)}\int_{[\ell_n, A-\lambda)}\lambda f(n,u+\lambda)\,dG_{\ell}(n, u)\Pi(n-1,d\lambda),
    \end{split}
    \ee
    where $K_0=\sum_{n=1}^\infty\int_{[\lambdao,A)}\lambda \bar{G}_{\ell}(n, (A-\lambda) \minus)\Pi(n-1,d\lambda)$.

    Choosing a test function $f$ such that $f(n,\lambda)=0$, for $n>0$, we have
    \be\label{eq:Stat_d_alpha0.0}
    \begin{split}
      \int_{[\lambdao,A)} &\lambda f(0,\lambda) \Pi(0,d\lambda)
      = f(0,\lambdao)K_0
    \end{split}
    \ee
    and, by arbitrariness of $f(0,\cdot)$, we get
    $
    \Pi(0, d \lambda)=\frac{K_0}{\lambdao} \delta_{\lambdao}(d\lambda).
    $

    Now we assume that $f(n,\lambda)=0$ for $n\not=1$ and we look for $\Pi(1,d\lambda) = \pi(1,\lambda)d\lambda$. So, Eq. \eqref{eq:Stat_d_alpha0} becomes:
    \be
    \begin{split}
      \int_{[\lambdao,A)} \lambda f(1,\lambda) \Pi(1,d\lambda)
      = & \int_{[\lambdao,A)}\int\ind_{\{u+\lambda<A\}}\lambda f(1,u+\lambda)\,dG_{\ell}(1, u)\Pi(0,d\lambda)\\
      =&  K_0  \int\ind_{\{u+\lambdao<A\}}f(1,u+\lambdao)\,dG_{\ell}(1, u)\\
      =&  K_0  \int\ind_{\{\lambda<A\}}f(1,\lambda)\,dG_{\ell}(1, \lambda -\lambdao).
    \end{split}
    \ee
    The support of $\Pi(1,\cdot)$ is $[\lambdao +\ell_1, A)$ and it holds
    $$
    \int_{[\lambdao + \ell_1, A)}  f(1,\lambda) \left[  \lambda \Pi(1, d \lambda) - K_0\,dG_{\ell}(1, \lambda-\lambdao) \right] =0
    $$
    so, by the arbitrariness of $f(1,\cdot)$
    $$
    \Pi(1, d \lambda) =\frac{K_0}{\lambda}\,dG_{\ell}(1, \lambda-\lambdao).
    $$

    To handle the case $n>1$, we assume that $f(m,\lambda)=0$, for $m\not=n$. So, Eq. \eqref{eq:Stat_d_alpha0} reads:
    \begin{equation*}
      \begin{split}
        \int_{[\lambdao,A)} \lambda f(n,\lambda) \Pi(n,d\lambda)
        = & \int_{[\lambdao,A)}\int\ind_{\{u+\lambda<A\}}\lambda f(n,u+\lambda)\,dG_{\ell}(n, u)\Pi(n-1,d\lambda)\\
        = & \int_{[\lambdao,A)}\int\ind_{\{u+\lambda<A\}}f(n,u+\lambda)\,dG_{\ell}(n, u) \bar\Pi(n-1,d\lambda) \\
        = & \int_{[\lambdao,A)} f(n,\lambda) d\bar G_{\ell}(n, \lambda)
      \end{split}
    \end{equation*}
    where $\overline \Pi(n,d\lambda) := \lambda \Pi(n,d\lambda)$
    and $\overline G_{\ell}(n,\cdot) :=G_{\ell}(n,\cdot) \star \overline \Pi(n-1,\cdot)$.
    The support of $\Pi(n,\cdot)$ is $[\lambdao +\ell_n, A)$,
    and by arbitrariness of $f(n,\cdot)$ we find
    $$
    \Pi(n, d \lambda) =\frac{1}{\lambda}  d\overline G_{\ell}(n,\lambda) .
    $$
  
\subsection{Proof of Theorem \ref{thm:stat.dist.alpha>0.beta=lambda0.exp}}\label{app:thm:stat.dist.alpha>0.beta=lambda0.exp}
We exploit Theorem \ref{thm:stat.dist.alpha>0.beta=lambda0} and we immediately find
  $\Pi(0,\{\beta\}) = K_0/\beta$, and for $\lambda\in(\beta,A)$,
  $\Pi(n, d\lambda) = \pi(n,\lambda) d\lambda, n \ge 1$.

By~\eqref{e:stdistycase2.pi}, for $n\ge1$, we have
  \be\label{eq:rec}
  \begin{split}
    \pi(n,\lambda)
    &= \frac{1}{\widehat F(\lambda-\beta)} \int_{\lambda-\beta}^{A-\beta} \frac{1}{\alpha y} \widehat F(y) \widehat p(n-1,y) g(y) dy ,
  \end{split}
  \ee
  with $g(y)=\theta e^{-\theta y}$, and
  \begin{align}
    \widehat F(y) &:= F(y+\beta) e^{\beta/\alpha} = y^{1-\frac{\beta}{\alpha}} e^{-y/\alpha} \label{eq:F}\\
    \widehat p(n-1,y) &:= \frac{1}{g(y)} p(n-1,y+\beta)
    = \int_{[\beta,y+\beta)} u e^{\theta (u-\beta)} \Pi(n-1,du) \ .\label{eq:p}
  \end{align}

  In particular, since by~\eqref{def:p(n-1,lambda).beta=lambdao}, we have, for $\lambda\in[\beta,A)$ and $n>1$ that
  \be
  \begin{split}
    p(n,\lambda)
    &=  \int_{[\beta,\lambda)} u g(\lambda-u) \Pi(n,du) ,
  \end{split}
  \ee
  it follows that
  \begin{align}
    F(\lambda)   &= \widehat F(\lambda-\beta) e^{-\beta/\alpha} \\
    p(n,\lambda) &= \widehat p(n,\lambda-\beta) g(\lambda-\beta) \label{eq:p_hat_p}
  \end{align}

  Summing~\eqref{eq:rec} over $n\ge1$,
  with $\pi(\lambda)=\sum_{n\ge1} \pi(n,\lambda)$
  and $\widehat p(y)=\sum_{n\ge0} \widehat p(n,y)$, we have
  \begin{align}
    \pi(\lambda) &= \frac{1}{\widehat F(\lambda-\beta)} \int_{\lambda-\beta}^{A-\beta} \frac{1}{\alpha y} \widehat F(y)\widehat p(y)g(y)dy ,\label{eq:sum.rec}\\
    \widehat p(y) &=  \int_{[\beta,y+\beta)} u e^{\theta (u-\beta)} \Pi(du) \label{eq:sum.p}
  \end{align}
  and, since $\Pi(d\lambda) = \frac{K_0}{\beta} \delta_{\beta}(d\lambda) + \pi(\lambda) d\lambda$, we have
  \be\label{eq:p.hat}
  \begin{split}
    \widehat p(y)
    &= K_0 + \int_\beta^{y+\beta} u e^{\theta (u-\beta)}  \pi(u) du
  \end{split}
  \ee

  It follows that the continuous density part satisfies the integral equation
  \be\label{eq:pi.integral.eq}
  \begin{split}
    \pi(\lambda)
    =& \frac{K_0}{\widehat F(\lambda-\beta)} \int_{\lambda-\beta}^{A-\beta} \frac{1}{\alpha y} \widehat F(y) g(y) dy\\
    &+ \frac{1}{\widehat F(\lambda-\beta)} \int_{\lambda-\beta}^{A-\beta} \int_\beta^{y+\beta}
    \frac{1}{\alpha y} \widehat F(y)g(y) u e^{\theta (u-\beta)}  \pi(u) du  dy \ .
  \end{split}
  \ee

  Then, by~\eqref{e:odepi1start} with $\widehat p(0,\lambda) = K_0$ (recall Eq. \eqref{eq:p}) and by~\eqref{diff.eq.pi.n.2} for $n\ge2$,  we get, for $n\ge1$, that with $\lambda\in[\beta,A)$ and using Eq. \eqref{eq:p_hat_p},
  \be
  (\lambda - \alpha) \pi(n,\lambda) - \alpha (\lambda-\beta)  \pi'(n,\lambda) = g(\lambda-\beta) \widehat p(n-1,\lambda-\beta)
  \ee

  Summing over $n\ge1$, with $\pi(\lambda)=\sum_{n\ge1} \pi(n,\lambda)$ and $\widehat p(\lambda)=\sum_{n\ge0} \widehat p(n,\lambda)$, we have
  \be\label{eq:pi.diff.eq}
  (\lambda - \alpha) \pi(\lambda) - \alpha (\lambda-\beta) \pi'(\lambda) = g(\lambda-\beta) \widehat p(\lambda-\beta) \ .
  \ee

  Let $h(y) = \widehat F(y) \pi(y+\beta)$, since the factor $\widehat F(y)$ defined in~\eqref{eq:F} satisfies the relation  $\alpha y \widehat F'(y) = -(y+\beta-\alpha) \widehat F(y)$, we have
  \be\label{eq:Fpi}
  \begin{split}
    -\alpha y h'(y)
    &= \widehat F(y) g(y)  \widehat p(y)
    \ .
  \end{split}
  \ee

  By equation~\eqref{eq:p.hat}, we have
  \begin{align}
    \widehat p'(y)
    &= \frac{\theta(y+\beta)}{g(y)}\pi(y+\beta)
    = \frac{\theta (y+\beta) h(y)}{g(y)\widehat F(y)}  \ .
  \end{align}

  By differentiating~\eqref{eq:Fpi} with respect to $y$, we get
  \be
  \begin{split}
    -\alpha y h''(y)
    - \alpha h'(y)
    =&\left(\beta-\alpha+(1+\theta\alpha)y\right) h'(y)
    + \theta(y+\beta) h(y) \ .
  \end{split}
  \ee
  where in the last equality we again used~\eqref{eq:Fpi}.

  After rearranging, we get the second order linear differential equation
  \be\label{eq:second.order.diff.eq}
  \alpha y h''(y)
  +(\beta+(1+\theta\alpha) y) h'(y)
  + \theta(y+\beta) h(y)=0
  \ee

  Let $ h(y)=d(y) g(y)$, we have
  \be
  \begin{split}
    h'(y)
    &=  g(y)(d'(y) - \theta d(y)) \\
    h''(y)
    &=  g(y)(d''(y) - 2\theta d'(y) + \theta^2 d(y)) \ .
  \end{split}
  \ee

  After substituting the expressions for $h'(y)$ and $h''(y)$ in~\eqref{eq:second.order.diff.eq} and divinding by~$g(y)$, we get
  \be\label{eq:second.order.diff.eq.d.1}
  \alpha y d''(y) + (\beta+ (1- \theta \alpha) y ) d'(y) =0
  \ee
  that gives
  \be\label{eq:second.order.diff.eq.d.2}
  \ln'(d'(y)) = - \frac{\beta}{\alpha}\frac{1}{y} - \left(\frac{1}{\alpha}- \theta \right)
  \ee
  with solution
  \be\label{eq:d'.sol}
  \begin{split}
    d'(y)
    &= K_1 \frac{\widehat F(y)}{y g(y)}  \ .
  \end{split}
  \ee
  and finally
  \be\label{eq:d.sol}
  d(y) = K_2 + K_1 \int_0^y \frac{\widehat F(z)}{z g(z)} dz.
  \ee

  Eventually, substituting back, we have
  \be\label{eq:pi.explicit.with.constants}
  \begin{split}
    \pi(\beta+y)
    &= \frac{h(y)}{\widehat F(y)}
    =\frac{g(y)d(y)}{\widehat F(y)}
    = \frac{g(y)}{\widehat F(y)} \left( K_2 + K_1 \int_0^y \frac{\widehat F(z)}{z g(z)} dz \right)  \ .
  \end{split}
  \ee

  Since $\pi(A)=0$, we have that $K_2=-K_1 \int_0^{A-\beta} \frac{\widehat F(z)}{z g(z)} dz$, and we can rewrite~\eqref{eq:pi.explicit.with.constants} as
  \be\label{eq:pi.explicit.one.constant}
  \begin{split}
    \pi(\lambda) &= - K_1\frac{g(\lambda-\beta)}{\widehat F(\lambda-\beta)} \int_{\lambda-\beta}^{A-\beta} \frac{\widehat F(z)}{z g(z)} dz
    \ .
  \end{split}
  \ee

  To determine the constant $K_1$, by~\eqref{eq:pi.diff.eq}, we have
  \be
  \begin{split}
    \theta K_0
    =&  \lim_{\lambda \to \beta^+} \left((\lambda-\alpha)  \pi(\lambda)
    -\alpha (\lambda-\beta)  \pi'(\lambda)\right)
    = -\alpha K_1 
  \end{split}
  \ee
  that gives $K_1= - \theta K_0 /\alpha$. Substituting back in~\eqref{eq:pi.explicit.one.constant}, we have
  \be\label{eq:pi.explicit.final}
  \begin{split}
    \pi(\lambda)
    &= K_0 \frac{\theta }{\alpha} \frac{g(\lambda-\beta)}{\widehat F(\lambda-\beta)} \int_{\lambda-\beta}^{A-\beta} \frac{\widehat F(z)}{z g(z)} dz  \\
    &= K_0 \frac{\theta }{\alpha} \frac{1}{\lambda-\beta} \int_\lambda^A \left(\frac{u-\beta}{\lambda-\beta}\right)^{-\frac{\beta}{\alpha}} e^{-(\frac{1}{\alpha}-\theta)(u-\lambda)} du
    \ .
  \end{split}
  \ee

  From~\eqref{eq:pi.explicit.final} we have, as $\lambda \to A^{\minus}$ that
  \be
  \begin{split}
    \pi(\lambda)
    \approx -(A-\lambda) \pi'(A)
    =
    K_0 \frac{\theta }{\alpha} \frac{A-\lambda}{A-\beta}
  \end{split}
  \ee
  while as  $\lambda \to \beta^+$ that
  \be
  \pi(\lambda)
  \approx K_0 \frac{\theta}{\alpha} \cdot
  \begin{cases}
    K_2  \left(\lambda-\beta\right)^{-\left(1-\frac{\beta}{\alpha}\right)}    &  \alpha > \beta \\
    \frac{\alpha}{\beta }  \log \left(\frac{A-\beta }{\lambda -\beta }\right)      &  \alpha=\beta\\
    \frac{\alpha}{\beta-\alpha} &  \alpha \in (0,\beta)
  \end{cases}
  \ee
  where $K_2=\left(\Gamma\left(1-\frac{\beta}{\alpha}\right)-\Gamma \left(1-\frac{\beta}{\alpha},\left(\frac{1}{\alpha}-\theta\right)(A-\beta)\right)\right)/\left(\frac{1}{\alpha}-\theta\right)^{1-\frac{\beta}{\alpha}}$
  and with $\Gamma(a,z)=\int _z^{\infty} t^{a-1} e^{-t} \, dt$ being the incomplete Gamma function.
  It follows that $\pi(\lambda)$ is integrable on $[\beta,A]$ for all values of $\alpha>0$.

  Now we try to solve for $K_0$ using the normalization condition. We know that $\int_\beta^A \pi(\lambda) d\lambda = 1 - K_0/\beta$, so that
  \be
  \begin{split}
    1 - \frac{K_0}{\beta}
    &= K_0 \frac{\theta }{\alpha} \int_\beta^A  \frac{1}{\lambda-\beta} \int_\lambda^A \left(\frac{u-\beta}{\lambda-\beta}\right)^{-\frac{\beta}{\alpha}} e^{-(\frac{1}{\alpha}-\theta)(u-\lambda)} du \, d\lambda \\
    &= K_0 \frac{\theta }{\alpha} \int_\beta^A \int_\beta^u  \frac{1}{\lambda-\beta} \left(\frac{u-\beta}{\lambda-\beta}\right)^{-\frac{\beta}{\alpha}} e^{-(\frac{1}{\alpha}-\theta)(u-\lambda)} d\lambda \, du
  \end{split}
  \ee
  substitute $\lambda-\beta = (u-\beta) z$, it follows that
  $u-\lambda= (u-\beta) (1-z)$ and therefore $d\lambda = (u-\beta) dz$, so we have, for $\alpha\not=\theta$
  \be
  \begin{split}
    1 - \frac{K_0}{\beta}
    &= K_0 \frac{\theta }{\alpha} \int_\beta^A \int_0^1 z^{\frac{\beta}{\alpha}-1} e^{-(\frac{1}{\alpha}-\theta)(u-\beta)(1-z)} dz \, du \\
    &= K_0 \frac{\theta }{\alpha} \int_0^1 z^{\frac{\beta}{\alpha}-1} \int_\beta^A e^{-(\frac{1}{\alpha}-\theta)(u-\beta)(1-z)} du \, dz \\
    &= K_0 \frac{\theta }{\alpha} \int_0^1 z^{\frac{\beta}{\alpha}-1} \frac{1 -e^{-(\frac{1}{\alpha}-\theta)(A-\beta)(1-z)}}{(\frac{1}{\alpha}-\theta)(1-z)}   dz \\
    &=  K_0 \frac{\theta }{\alpha} K_3
  \end{split}
  \ee
  where
  \be\label{eq:K3.def}
  K_3=\int_0^1  \frac{z^{\frac{\beta}{\alpha}-1}}{1-z} \left(1 -e^{-(\frac{1}{\alpha}-\theta)(A-\beta)(1-z)}\right) dz \ .
  \ee

  It follows that
  \be\label{eq:K0}
  \Pi(\{\beta\})=\frac{K_0}{\beta} = \frac{\alpha}{\alpha+ \beta \theta K_3} \ .
  \ee

  Now,
  \be\label{eq:K3.rep}
  K_3 =
  \begin{cases}
    -{\Kummer}^{(1,0,0)}\left(0,\frac{\beta }{\alpha },\frac{(\alpha  \theta -1) (A-\beta )}{\alpha }\right), & \alpha \not= 1/\theta \\
    \frac{\alpha}{\beta}  \left( A -\beta\right) & \alpha = 1/\theta
  \end{cases}
  \ee
  where ${\Kummer}^{(1,0,0)}$ is the Kummer’s confluent hypergeometric function.

\end{appendix}

\end{document}